\documentclass[1 [leqno,10pt]{amsart}
\usepackage{amssymb}
\usepackage{amssymb, amsmath}
\usepackage{color}
\usepackage{verbatim}

\allowdisplaybreaks
\let\d=\partial
\let\eps=\varepsilon
\let\wt=\widetilde

\def\cC{{\mathcal C}}

\def\cP{{\mathcal P}}
\def\cQ{{\mathcal Q}}

\def\cS{{\mathcal S}}
\def\cT{{\mathcal T}}

\def\N{{\mathbb N}}
\def\R{{\mathbb R}}

\def\Z{{\mathbb Z}}

\def\virgp{\raise 2pt\hbox{,}}
\def\cdotpv{\raise 2pt\hbox{;}}
\def\Id{\mathop{\rm Id}\nolimits}

\def\div{ \hbox{\rm div}\,  }

\def\ddj{\dot \Delta_j}

\def\na{\nabla}
\def\lam{\lambda}
\def\Lam{\Lambda}
\def\pa{\partial}

\def\f{\frac}

\def\eqdefa{\buildrel\hbox{\footnotesize def}\over =}

\newtheorem{thm}{Theorem}[section]
\newtheorem{lem}{Lemma}[section]
\newtheorem{rmk}{Remark}[section]

\newtheorem{prop}{Proposition}[section]

\newcommand{\ben}{\begin{eqnarray}}
\newcommand{\een}{\end{eqnarray}}
\newcommand{\beno}{\begin{eqnarray*}}
\newcommand{\eeno}{\end{eqnarray*}}

\begin{document}
\title[On $\mbox{(CNS)}$ in $L^p$ type critical spaces]
{Global stability of large solutions to the 3D compressible Navier-Stokes equations}

\author[L. He]{Lingbing He}
\address[L. He]{Department of Mathematical Sciences, Tsinghua University\\
Beijing 100084,  P. R.  China.} \email{lbhe@math.tsinghua.edu.cn}
\author[J. Huang]{Jingchi Huang}
\address[J. Huang]{ School of Mathematics, Sun Yat-sen University, Guangzhou Guangdong 510275, P. R. China. } \email{huangjch25@mail.sysu.edu.cn}
 \author[C. Wang]{Chao Wang}
\address[C. Wang]{ School of Mathematical Sciences, Peking University\\
 Beijing 100871,  P. R.  China.} \email{wangchao@math.pku.edu.cn }

\begin{abstract}
The present paper aims at the investigation of the global stability of large solutions to the compressible Navier-Stokes equations in the whole space. Our main results and innovations  can be concluded as follows:
\smallskip

 \noindent $\bullet$ Under the assumption that the density $\rho(t,x)$ verifies  $\rho(0,x)\ge c>0$ and $\sup_{t\ge0}\|\rho(t)\|_{C^\alpha}\le M$ with $\alpha$ sufficiently small,  we establish a new mechanism for the convergence of the solution  to its associated equilibrium  with  an explicit decay rate  which is as the same as that for the heat equation.  The main idea of the proof relies on the basic energy identity,  techniques from blow-up criterion and a new estimate for the low frequency part of the solution.
 
 \noindent $\bullet$ We prove the global-in-time stability for the equations, i.e,  any perturbed solution will remain close to the reference solution if initially they are close to each other.  Our result implies that the set of the smooth and bounded solutions is an open set.

\noindent $\bullet$ Going beyond the close-to-equilibrium setting, we construct the global large solutions to the equations with a class of initial data in $L^p$ type critical spaces.  Here the ``large solution" means that the vertical component of the  velocity could be arbitrarily large initially.

\end{abstract}

 \maketitle

\section{introduction}
In this paper, we are concerned with the global stability of the large solutions to  3-D barotropic compressible Navier-Stokes equations:
$$
 \left\{\begin{array}{l}
\partial_t\rho+{\rm div}(\rho  u )=0,\\[0.5ex]\displaystyle
\partial_t(\rho  u)+{\rm div} (\rho  u \otimes u )
-\div\bigl(2\mu(\rho )D(u )+\lambda(\rho )\div u\Id\bigr)+\nabla P=0,\\
 \lim\limits_{|x|\rightarrow \infty} \rho=1,
\end{array}
\right.\leqno(CNS)
$$
where $\rho=\rho(t,x)\in\R^+$ stands for the  density,
$u =u(t,x)\in\R^3$ is the  velocity field and  the pressure $P$ is given by smooth function $P=P(\rho)$. Here we take $P(\rho)=\rho^\gamma$ for $\gamma\geq1$. The bulk and shear viscosities are given by
$\lambda=\lambda(\rho)$ and $\mu=\mu(\rho)$ which satisfy 
$
\mu>0$ and $\lambda+2\mu>0.
$
In our paper, we assume that $\mu$ and $\lambda$ are two constants. Finally, $D(u)$ stands for the deformation tensor, that is $(D(u ))_{ij}:=\frac12(\d_iu^{j}+\d_ju^{i}).$

\subsection{Short review of the system $\mbox{(CNS)}$}
 There  are huge number of literatures on the barotropic compressible Navier-Stokes equations. Here we only review some results which are related to our stability result.

\subsubsection{Well-posedness results in Sobolev spaces} The local well-posedness for the system $\mbox{(CNS)}$
was proved by Nash \cite{Nash} for the smooth initial data which is away from vacuum. For the global smooth solutions, it was first proved by Matsumura and Nishida in \cite{MN1, MN2}
if the initial data is close to equilibrium in $H^3\times H^3$. For the small energy, Zhang \cite{Zhang} and Huang, Li and Xin \cite{Huang} proved the global existence and uniqueness of $\mbox{(CNS)}$. For the general initial data, if $\mu=const, \lambda(\rho)=b \rho^\beta$, the authors in \cite{Huang1, VK} established the global existence and uniqueness of classical solutions for large initial data   in two dimension.  

\subsubsection{Well-posedness results in Critical spaces}  To catch the scaling invariance property of the system $\mbox{(CNS)}$, Danchin first introduced  in his series papers \cite{Dan-Inve,Dan-CPDE, Dan-ARMA, Dan-CPDE07} the ``Critical Spaces'' which were inspired by the results for  the incompressible Navier-Stokes.  More precisely, he proved the local well-posedness of  $\mbox{(CNS)}$ in the critical Besov spaces $\dot B^{\f 3p}_{p,1}\times \dot B^{\f3p-1}_{p,1}$ with $2\le p<6$,
and global well-posdenss of  $\mbox{(CNS)}$ for the initial data close to a stable equilibrium in spaces
$(\dot B^{\f12}_{2,1}\cap \dot B^{\f32}_{2,1}) \times \dot B^{\f12}_{2,1} $. We mention that when $p>3$, the Besov space $\dot B^{\f3p-1}_{p,1}$  contains the data which allows to have high oscillation. A typical example is
\begin{align*}
u_0(x)= \sin\bigl(\f {x_3} {\varepsilon}\bigr)(-\pa_2\phi(x), \pa_1\phi(x),0)
\end{align*}
where $\phi\in \cS(\R^3)$ and $\varepsilon>0$ is small enough.  It is easy to check that $\|u_0\|_{\dot B^{\f3p-1}_{p,1}}\sim \varepsilon^{1-\f3p},$ but $\|u_0\|_{\dot H^{\f12}}\sim \varepsilon^{-\f12}$ can be arbitrarily large. The authors in \cite{Charve, CMZ} constructed global solutions of  $\mbox{(CNS)}$ with the highly oscillating initial data. Later, Haspot in \cite{Has} gave an alternative proof to the similar result by using the viscous effective flux. Very recently, the authors in \cite{Dan-He} generalized the previous results to allow that the incompressible part of the velocity belongs to $\dot B^{\f 3p-1}_{p,1}$ with $p\in[2,4]$. In this case, by taking the low mach number limit, the solutions for $\mbox{(CNS)}$ will converge to the solutions for the incompressible Navier-Stokes equations  constructed by Cannone, Meyer and Planchon in \cite{Can, CMP}.    Finally we mention the work \cite{FZZ} on the construction of the large solutions to  $\mbox{(CNS)}$ based on the dispersion property of acoustic waves.
 
\subsubsection{Results on global dynamics and the stability}  To our best knowledge, almost all the results on the global dynamics  and the stability  are restricted to the regime which is near the equilibrium.  Thus the method relies heavily on the analysis of the linearization of the system and   the standard perturbation framework.   We refer readers to  \cite{DUYZ, Kobayashi1,Kobayashi2,Kobayashi3,Kobayashi4, MN1, MN2,Ponce} and reference therein for   details. 

 Assume that the initial data $(\rho_0, u_0)$ is a small perturbation of equilibrium $(\rho_\infty, 0)$ in $L^p(\R^3)\times H^3(\R^3)$ with $p\in[1,2]$.  Previous results show that \ben\label{decay}
\|(\rho-\rho_\infty, u)(t)\|_{L^2}\leq C(1+t) ^{-\f34(\f2{p}-1)}.
\een
It discloses that in the close-to-equilibrium setting the   rate of the convergence of  the solution is as the same as that for the heat equations if the initial data is in $L^p(\R^3)$.  In this sense, \eqref{decay} is the optimal decay estimate for system $\mbox{(CNS)}$. Recently, Danchin and Xu in \cite{DX} generalized this estimate  in the critical $L^p$ framework based on the existence result established in \cite{Dan-He}.

Beyond the close-to-equilibrium regime, there exist two results on the longtime behavior of the solution of $\mbox{(CNS)}$ for the general initial data   in bounded domains. The first result is due to Fang, Zhang and Zi. In \cite{FZZ1}, they showed that the weak solutions constructed by P.-L. Lions and improved by E. Feireisl  decayed exponentially to equilibrium in $L^2$ space, under homogeneous Dirichlet boundary conditions and the condition that the density is bounded from above.  The key idea of the proof lies in the Poincare inequality, the energy identity and the Bogovskii operator $\mathfrak{B}$ to get the integrability of $\rho$. The second result is due to Villani. By using the hypo-coercivity, in \cite{villani}, he proved that if the solution remains smooth and bounded in the torus, then it will converge to its associated equilibrium  with  algebraic rate.

Unfortunately both of methods used in \cite{FZZ1} and \cite{villani} cannot be applied to get the strong stability in the whole space. We first note that to get the longtime behavior of the solution both methods rely more or less on the fact that domain is bounded. For instance,  the Poincare inequality  and the $L^p$ bound of  Bogovskii operator $\mathfrak{B}$  are used which do not hold in the whole space. Secondly the propagation of the regularity is not considered in both papers which is essential to prove the global-in-time stability. These show that previous work cannot help a lot to prove the desired result and we need some new idea.

\subsection{Main idea and strategy}  The present work aims at  the investigation of  the  global-in-time stability of the large solutions to $\mbox{(CNS)}$.  The key point to solve the problem lies in the global dynamics of the system $\mbox{(CNS)}$:  the propagation of the regularity and the mechanism for the convergence to the equilibrium. 

Our strategy is carried out by three steps:  getting   uniform-in-time bounds for the propagation of the regularity, deriving a dissipation inequality  and using time-frequency splitting method to obtain a new mechanism for the convergence to the equilibrium with quantitative estimates. 

  To obtain uniform-in-time bounds for the solution, we borrow some techniques from the blow-up criterion for the system used in \cite{Huang,Huang1,SWZ1,SWZ2}. There the authors proved that the upper bound of the density will control the propagation of the regularity.  But they failed to get the uniform-in-time  bounds.  Under the assumption  
  that the density is bounded  uniformly in time in $C^\alpha$ with  $\alpha $ sufficiently small, which is a little stronger than  the assumption used in \cite{Huang,Huang1,SWZ1,SWZ2}, 
we succeed to prove the  uniform-in-time bounds  for the propagation of the regularity. We remark that here the key idea is making full use  of  the basic energy identity and the coupling effect of the system.  

Each time when    the uniform-in-time bounds  for the  regularity  of the solution are improved,  the  dissipation inequality can also be improved correspondingly. Thanks to this observation, finally we obtain that
\beno \f{d}{dt} \|(\rho-1, u)\|^2+\|\na (\rho-1, u)\|^2\le 0,  \eeno
which enjoys the same structure as that for the heat equation. 
Now the time-frequency splitting method is evoked.  More precisely, the dissipation inequality in the above can be recast by 
 \ben\label{ds1}  \f{d}{dt} \|(\rho-1, u)\|^2+\f{1}{1+t}\|  (\rho-1, u)\|^2\le \f{1}{1+t} \int_{|\xi|\le (1+t)^{-\f12} }  \big|(\widehat{\rho-1}, \hat{u})(\xi)\big|^2
 . \een 
The problem of the convergence  is reduced to the estimate of the low frequency part of the solution which is easy to do it for the linear equation. For the non-linear equation in the regime which is far away from the  equilibrium, our new idea is   making full use of the cancellation and the coupling effect of the system $\mbox{(CNS)}$  to get the control of the righthand side of \eqref{ds1}(see Lemma \ref{conlf}).
As a result, we obtain the optimal decay estimate \eqref{decay}.  
We comment that the control of the righthand side of \eqref{ds1}  is comparable to the one due to Schonbek in \cite{Schon}  for the incompressible NS equations. It is robust considering that we only request that the density is bounded from above uniformly in time.  

\smallskip

Once the global dynamics of the equations is clear, we can prove the global-in-time stability  for the system  $\mbox{(CNS)}$. The strategy  falls into three steps:
\begin{enumerate}
	\item By the local well-posedness for the system $\mbox{(CNS)}$,  we can show that the perturbed solution will remain close to the reference solution for a long time if initially they are close. 
	\item The  mechanism for the convergence implies that the reference solution is close to the  equilibrium after a long time. 
	\item  Combining these two facts, we can find a time $t_0$ such that $t_0$ is far away from the initial time and  at this moment the solution is close to the equilibrium. Then it is not difficult to prove the global existence  in the perturbation framework.
	\end{enumerate}

\smallskip

To show that our result on the global-in-time stability has wide application, we construct the large solutions to the system $\mbox{(CNS)}$ with the initial data  in some $L^p$ critical spaces. 
Our main idea is inspired by \cite{PZ}. The main observation lies in  two aspects. The first one is that  the equation for $(\cP u)^3$, the vertical component of the incompressible part of the velocity $u$, is actually a linear equation. The second one is that there is no quadratic term for $(\cP u)^3$ in the nonlinear terms.  Motived by these two facts, we construct a global solution for compressible Navier-Stokes equations with the initial data such that  $(\cP u_0)^3$ could be arbitrarily large.  Obviously such kind of the solution  initially is far away from the equilibrium.

\subsection{Function spaces and main results}\label{s:main} Before we state our results, let us introduce the notations and  function spaces which will be used throughout the paper.  
   We denote the multi-index $\alpha =(\alpha _1,\alpha _2,\alpha _3)$ with
$|\alpha |=\alpha _1+\alpha _2+\alpha _3$.   We use the notation $a\sim b$ whenever $a\le C_1 b$ and $b\le C_2
a$ where $C_1$ and $C_2$ are universal constants.    We denote $C(\lambda_1,\lambda_2,\cdots, \lambda_n)$ by a constant depending on   parameters $\lambda_1,\lambda_2,\cdots, \lambda_n$. If  $u=(u^1,u^2,u^3)$, then we set $u^h\eqdefa (u^1,u^2)$.

 We recall that an homogeneous Littlewood-Paley decomposition
$(\ddj)_{j\in\Z}$ is a dyadic decomposition in the Fourier space for $\R^d.$
One may for instance set $\ddj:=\varphi(2^{-j}D)$ with $\varphi(\xi):=\chi(\xi/2)-\chi(\xi),$
and $\chi$ a non-increasing nonnegative smooth function supported in $B(0,4/3),$ and with
value $1$ on $B(0,3/4)$  (see \cite{Bah}, Chap. 2 for more details).
We then define, for  $1\leq p,r\leq \infty$  and $s\in\R,$ the semi-norms
$$
\|z\|_{\dot B^s_{p,r}}:=\bigl\|2^{js}\|\ddj z\|_{L^p(\R^d)}\bigr\|_{\ell^r(\Z)}.
$$
Like in \cite{Bah},  we adopt the following definition of homogeneous
Besov spaces, which turns out to be well adapted to the study of nonlinear PDEs:
$$
\dot B^s_{p,r}=\Bigl\{z\in\cS'(\R^d) : \|z\|_{\dot B^s_{p,r}}<\infty \ \hbox{ and }\ \lim_{j\to-\infty}
\|\dot S_jz\|_{L^\infty}=0\Bigr\}
\quad\hbox{with }\ \dot S_j:=\chi(2^{-j}D).$$
As we shall work with \emph{time-dependent functions} valued in Besov spaces,
we  introduce the norms:
$$
\|u\|_{L^q_T(\dot B^s_{p,r})}:=\bigl\| \|u(t,\cdot)\|_{\dot B^s_{p,r}}\bigr\|_{L^q(0,T)}.
$$
As  pointed out in \cite{Chemin-JDE},  when  using parabolic estimates in Besov spaces, it is somehow natural
to take the time-Lebesgue norm \emph{before} performing the summation for computing
the Besov norm.  This motivates our introducing the following quantities:
$$
 \|u\|_{\wt L_T^q(\dot {B}^{s}_{p,r})}:= \bigl\|(2^{js}\|\ddj u\|_{L_T^q(L^p)})\bigr\|_{\ell^r(\Z)}.
$$
The index $T$ will be omitted if $T=+\infty$ and we shall denote by $\wt\cC_b(\dot B^s_{p,r})$
 the subset of functions of  $\wt L^\infty(\dot B^s_{p,r})$
which are also  continuous from  $\R_+$ to $\dot B^s_{p,r}$.

Let us emphasize that, owing to Minkowski inequality, we have if  $r\le q$,
$$
 \|z\|_{L_T^q(\dot {B}^{s}_{p,r})}\leq \|z\|_{\wt L_T^q(\dot {B}^{s}_{p,r})}
$$
with equality if and only if $q=r.$ Of course, the opposite inequality occurs if $r\geq q.$
\medbreak
An important  example where those nonclassical norms  are  suitable is the heat equation
\begin{equation}\label{eq:heat0}
\d_tz-\mu\Delta z=f,\qquad z|_{t=0}=z_0
\end{equation}
for which the following family of inequalities holds true:
\begin{equation}\label{eq:heat}
 \|z\|_{\wt L_T^m(\dot {B}^{s+2/m}_{p,r})}\leq C\bigl(\|z_0\|_{\dot B^s_{p,r}}
 + \|f\|_{\wt L_T^1(\dot {B}^{s}_{p,r})}\bigr)
\end{equation}
for any $T>0,$ $1\leq m,p,r\leq\infty$ and $s\in\R.$

\medbreak

  Now we are in a position to state our main results on the system $\mbox{(CNS)}$. 
Our first result is concerned with  the global dynamics of the equation.
\begin{thm}\label{thm:main1}
Let $\mu>\frac12 \lambda,$ and $(\rho,u)$ be a global and smooth  solution of $\mbox{(CNS)}$ with initial data $(\rho_0,u_0)$ where $\rho_0\geq c>0.$ Suppose the admissible condition  holds:
\ben\label{admissible}
u_t|_{t=0}=-u_0\cdot\na u_0+\f1{\rho_0}Lu_0-\f1{\rho_0}\na \rho^\gamma_0,
\een
where operator $L$ is defined by $Lu=-\div(\mu\na u)-\na((\lambda+\mu)\div u).$ Assume that  $a\eqdefa \rho-1$,  and $\|\rho\|_{C^\alpha}\leq M$ for some $0<\alpha<1.$ Then if $a_0, u_0\in L^{p_0}(\R^3)\cap H^2(\R^3)$ with $p_0\in [1,2)$, we have
\begin{enumerate}
\item {\bf (Lower bound of the density) } There exists a positive constant $ \underline{\rho}=\underline{\rho}(c,M)$ such that for all $t\ge0$,  $\rho(t)\ge \underline{\rho}$.
\item {\bf (Uniform-in-time bounds for the regularity of the solution)}
\beno \|a\|_{L^\infty_t H^2}^2+\|u\|_{L^\infty_t H^2}^2+\int_0^\infty (\|a\|_{ H^2}^2+\|\na u\|_{  H^2}^2 )d\tau\le C(\underline{\rho}, M, \|a_0\|_{L^{p_0}\cap H^2},\|u_0\|_{L^{p_0}\cap H^2} ).\eeno
\item {\bf (Decay estimate for the solution)}
\begin{equation}\label{decayest}
\|u(t)\|_{H^1} + \|a(t)\|_{H^1} \leq C(\underline{\rho}, M, \|a_0\|_{L^{p_0}\cap H^1},\|u_0\|_{L^{p_0}\cap H^2} )(1+t)^{-\beta({p_0})}, \end{equation}
where  $\beta({p_0})=\f34(\f2{p_0}-1)$.
\item{\bf (Decay estimates in  Critical spaces and the control of $\|\na u\|_{L^1((0,\infty); L^\infty)}$) }  
\ben
&&  \|u(t)\|_{\dot{B}^{\f12}_{2,1}} +\|a(t)\|_{\dot{B}^{\f12}_{2,1}}\leq C(\underline{\rho}, M, \|a_0\|_{L^{p_0}\cap H^2},\|u_0\|_{L^{p_0}\cap H^2} )(1+t)^{-\beta({p_0})}, \label{decayest1}\\
&& \|a(t)\|_{ \dot{B}^{\f32}_{2,1}} \leq C(\underline{\rho}, M, \|a_0\|_{L^{p_0}\cap H^2},\|u_0\|_{L^{p_0}\cap H^2} )(1+t)^{-\f12\beta({p_0})},\label{decayest4}\\
&&\int_0^\infty \|\na u\|_{L^\infty} dt \le C(\beta({p_0}), \underline{\rho}, M, \|a_0\|_{L^{p_0}\cap H^2},\|u_0\|_{L^{p_0}\cap H^2}).
\label{lip}
\een

\end{enumerate}
 \end{thm}

Several remarks are in order:
\begin{rmk}   We fail to quantify the dependence of $\underline{\rho}$ on  $c$ and $M$. The reason results from the  fact that $\lim\limits_{t\rightarrow\infty} \|\rho^\gamma(t)-1\|_{L^6}=0$ is used in the proof,  which shows that the lower bound of $\rho$ depends on the solution itself.  \end{rmk}

\begin{rmk} Once the constants $\underline{\rho}$ and $M$ are fixed, our theorem shows that all the upper bounds obtained in the theorem depend only on the initial data.  \end{rmk}

\begin{rmk}  Since \eqref{decayest} implies   \eqref{decay}, our decay estimate is optimal in some sense. The control of $\|\na u\|_{L^1((0,\infty); L^\infty)}$ implies that regularity of the solution can be propagated uniformly in time.
\end{rmk}	
 \begin{rmk}
The  method employed here is robust and can be  applied to the case that  $\mu$ and $\lambda$ are smooth functions and to the full Navier-Stokes equations. 
\end{rmk}

Next we want to state our global-in-time stability result for the system  $\mbox{(CNS)}$.  To prove the result in the largest function space, we solve the problem in   critical spaces. From now on, we agree that for $z\in\cS'(\R^d),$
\begin{equation}
\label{eq:decompo}  z^{\ell }:=\sum_{2^j \leq R_0} \ddj z\quad\mbox{and}\quad z^{H }:=\sum_{2^j >R_0} \ddj z,
\end{equation}
 for some large enough nonnegative integer $R_0$  depending only on $p,$ $d$.
 The corresponding  ``truncated'' semi-norms are defined  as follows:
\beno \|z\|^{\ell }_{\dot B^{\sigma}_{p,r}}:=  \|z^{\ell }\|_{\dot B^{\sigma}_{p,r}}
\ \hbox{ and }\   \|z\|^{H}_{\dot B^{\sigma}_{p,r}}:=  \|z^{H}\|_{\dot B^{\sigma}_{p,r}}. \eeno
To simplify the notation, we introduce the notiaion:
 \beno
 \|u\|_{\dot B^{s,t}_{r,p}}\eqdefa \|u\|^{\ell}_{\dot B^{s}_{r,1}}+ \|u\|^H_{\dot B^{ t}_{p,1}}.
 \eeno

Our second result can be stated as follows:
\begin{thm}\label{thm:main2}
Let $(\bar\rho, \bar u)$ be a global and smooth solution for the \mbox{(CNS)} with the initial data $(\bar \rho_0, \bar u_0)$ verifying that   \ben\label{assm}
  \|\f{1}{\bar \rho}, \bar \rho, \na \bar\rho \|_{\wt  L^\infty_t(\dot B^{\f12,\f3p}_{2,p}) }+\|\bar u, \bar u_t\|_{\wt  L^\infty_t(\dot B^{\f12,\f3p-1 }_{2,p})}+\|\bar u\|_{\wt L^\infty_t(\dot B^{\f52,\f3p+1}_{2,p})}\leq C, 
  \een
  where $2\leq p\leq4$. Assume that $(\bar\rho_0-1, \bar u_0)\in L^{p_0}(\R^3)\cap H^2(\R^3)$ with $p_0\in [1,2)$. There exists a $\eps_0=\eps_0(C)$ depending only on $C$ such that for any $0<\eps\leq\eps_0$, if
\ben\label{assm1}
\|(\rho_0-\bar\rho_0)(t)\|_{\dot B^{\f12,\f3p}_{2,p}}+\|\cP(u_0-\bar u_0)(t)\|_{\dot B^{ \f3p-1}_{p,1}}+\|\cQ (u_0-\bar u_0)(t)\|_{\dot B^{\f12,\f3p-1}_{2,p}}\leq \eps,
\een
then $\mbox{(CNS)}$ admits a global  and unique solution $(\rho, u)$ with the initial data $(\rho_0, u_0)$. Moreover, for any $t>0$,  
\beno
&&\|(\rho-\bar\rho)(t)\|_{\dot B^{\f12,\f3p}_{2,p}}+\|\cP(u-\bar u)(t)\|_{\dot B^{ \f3p-1}_{p,1}}+\|\cQ (u-\bar u)(t)\|_{\dot B^{\f12,\f3p-1}_{2,p}} \\
&&\lesssim \min\{(1+\delta|\ln\eps|)^{-\beta(p_0)/2}, (1+t)^{-\beta(p_0)/2}+\epsilon\},  \eeno
where  $\delta$ is a constant independent of $\eps,$ $\beta(p_0)$ is defined in Theorem \ref{thm:main1}, $\cP=I+\na (-\Delta)^{-1}\div,$ and $\cQ=-\na (-\Delta)^{-1}\div$.

\end{thm}

\begin{rmk}
The assumption \eqref{assm} on the reference solution is easily satisfied if the solution is in $H^s$ with $s\geq 4$. 
\end{rmk}

\medbreak

Finally we want to construct some   solutions to the system $\mbox{(CNS)}$  which initially are far away from the  equilibrium. 
 We will deal with the initial data $(a_0,u_0)$  verifying that 
\begin{itemize}
\item $(a_0,\cQ u_0)^{\ell } \in\dot B^{\f12}_{2,1},$
\item $(a_0)^{H}\in \dot B^{\f3p}_{p,1},\quad (\cQ u_0)^{H}\in \dot B^{\f3p-1}_{p,1},$
\item $\cP u_0\in \dot B^{\f3p-1}_{p,1} .$
\end{itemize}

  \begin{thm}\label{thm:main3}
Let $2\leq p\leq 4$.   If  $(a_0 , u_0 )$ satisfies that
\begin{equation}\label{eq:smalldata1}
\big(\|(a_0,\cQ u_0)\|^{\ell}_{\dot B^{\f12}_{2,1}}+\|\cQ u_0 \|^{H}_{\dot B^{\f3p-1}_{p,1}}+\|a_0 \|^{H}_{\dot B^{\f3p}_{p,1}}+\|(\cP u_0)^h\|_{\dot B^{\f3p-1}_{p,1}}\big)\exp \big(C(1+\|(\cP u_0)^3\|_{\dot B^{\f3p-1}_{p,1}} )\big)\leq \eps,
\end{equation}
where $C$ is a universal constant,
then the system $\mbox{(CNS)}$ with initial data $(a_0 , u_0 )$ admits a unique and global
solution $(a,u)$ satisfying that for all $t>0,$
\beno
&&\|a(t)\|_{\dot B^{\f12, \f3p}_{2,p}}+ \|  \cQ u(t)\|_{\dot B^{\f12, \f3p-1}_{2,p}}+ \|(\cP u)^h(t)\|_{\dot B^{  \f3p-1}_{p,1}}\\
&&\quad+\int^{\infty}_0\bigg(\|a\|_{\dot B^{\f52, \f3p}_{2,p}}+ \|  \cQ u\|_{\dot B^{\f52, \f3p+1}_{2,p}}+ \|(\cP u)^h \|_{\dot B^{ \f3p+1}_{p,1}}\bigg)dt \leq C \bigg(1+\|(\cP u_0)^3\|_{\dot B^{ \f3p-1}_{p,1}}\bigg)\eps,
\eeno
and
\beno
\|(\cP u)^3(t) \|_{\dot B^{ \f3p-1}_{p,1}}+\int_0^{\infty}\|(\cP u)^3 \|_{\dot B^{ \f3p+1}_{p,1}}dt\leq 2\|(\cP u_0)^3 \|_{\dot B^{ \f3p-1}_{p,1}}+\eps.
\eeno
  \end{thm}

\subsection{Organization of the paper}  We first give a rigorous proof to Theorem  \ref{thm:main1}, the global dynamics of $\mbox{(CNS)}$ in Section 2. Then we will prove the global-in-time stability for $\mbox{(CNS)}$, in Section 3. In the next section, Section 4, we will construct global large solutions with a class of initial data in $L^p$ critical spaces. Last, we list some basic knowledge on Littlewood-Paley decomposition in the appendix.

\setcounter{equation}{0}
\section{Global dynamics of the compressible Navier-Stokes equations}
In this section, we will give the proof to Theorem \ref{thm:main1}. To do that, we separate the proof into two steps: 
getting   uniform-in-time bounds  and  the dissipation inequality first and then applying the time-frequency splitting method to obtain the convergence to the equilibrium with quantitative estimates.

\subsection{Uniform-in-time bounds and the dissipation inequality}   In what follows, we will set $a\eqdefa\rho-1$ and
$\mathfrak{a}\eqdefa\rho^\gamma-1.$  Observe that $\mathfrak{a}=\big(\int_0^1 \gamma (\theta\rho+(1-\theta))^{\gamma-1}d\theta\big) a$. Thus if $\rho\le\overline{\rho} $, \ben\label{afraka}|\mathfrak{a}|\sim |a|. \een 

We begin with two lemmas. We first recall the basic energy identity for $\mbox{(CNS)}$:
\begin{lem}\label{energylem1}
	Let $(\rho, u)$ be a global and smooth solution of $\mbox{(CNS)}$. Then the following equality holds
	\begin{equation}\label{energyest2}
	\frac{d}{dt}\Big( \int H(\rho|1)\,dx+\f12\int \rho u^2\,dx\Big)+\mu \|\na u\|_{L^2}^2+(\lam+\mu)\|\div u\|_{L^2}^2
	=0,
	\end{equation}
	where
	\begin{equation*} 
		H(\rho|1)=
 \left\{\begin{array}{l}
\displaystyle \f{1}{\gamma-1}(\rho^\gamma-1-\gamma(\rho-1)),\qquad \textrm{when}\quad \gamma>1,\\
\displaystyle \rho\ln\rho-\rho+1 , \qquad \textrm{when}\quad \gamma=1.
\end{array}\right.
\end{equation*}
\end{lem}
\begin{rmk}\label{frakaequa} 
	By Taylor expansion, it is not difficult to check that
$H(\rho|1)\geq C(\overline{\rho})(\rho-1)^2$  if $\rho\leq \overline{\rho}$.
\end{rmk}

\begin{lem}\label{energylem2}
	Let $\mu>\f12 \lam,$ and $(\rho, u)$ be a global solution of $\mbox{(CNS)}$ with $0\leq \rho\leq \bar{\rho}.$ Then the following inequality holds
	\begin{equation}\label{energyest3}
	\frac{d}{dt}\int \rho u^4\,dx+ \int |u|^2|\na u|^2\,dx
	\leq  C\|\na u\|_{L^2}^2( \| \mathfrak{a}\|_{L^6}^2 +\|\na u\|_{L^2}^2),
	\end{equation}
	where $C$ is a positive constants depending on $\mu$ and $\lam$. 
\end{lem}
\begin{proof}
	Multiplying $4|u|^2 u$ to the second equation of $\mbox{(CNS)}$, and then integrating on $\R^3,$ we can obtain that
	\beno
	\begin{split}
		\frac{d}{dt}&\int \rho u^4\,dx + \int \Big[4|u|^2\Big( \mu|\na u|^2+(\lam+\mu)(\div u)^2+2\mu\big|\na|u|\big|^2  \Big)+4(\lam+\mu)(\na |u|^2)\cdot u\div u    \Big]\,dx\\
		&=4\int \div (|u|^2 u) \mathfrak{a}\,dx\leq C\int \mathfrak{a} |u|^2|\na u|\,dx
		\leq C\|\mathfrak{a}\|_{L^6}\|u^2\|_{L^3}\|\na u\|_{L^2}\leq C\|  \mathfrak{a}\|_{L^6}^2\|\na u\|_{L^2}^2+\|\na u\|_{L^2}^4.
	\end{split}
	\eeno
	Using the inequality $\big|\na|u|\big|\leq |\na u|$, we have
	\beno
	\begin{split}
		&4|u|^2\Big( \mu|\na u|^2+(\lam+\mu)(\div u)^2+2\mu\big|\na|u|\big|^2  \Big)+4(\lam+\mu)(\na |u|^2)\cdot u\div u\\
		\geq& 4|u|^2\Big[ \mu|\na u|^2+(\lam+\mu)(\div u)^2+2\mu\big|\na|u|\big|^2  -2(\lam+\mu)\big|\na|u|\big| |\div u|\Big]\\
		=& 4|u|^2\Big[ \mu|\na u|^2+(\lam+\mu)\Big( \div u- \big|\na|u|\big|\Big)^2\Big]+4|u|^2(\mu-\lam)\big|\na|u|\big|^2\\
		\geq&4(2\mu-\lam)|u|^2\big|\na |u|\big|^2\geq C|u|^2|\na u|^2,
	\end{split}
	\eeno
	where in the last step we use $\mu>\f12\lam.$ Combining these two estimates, we arrive at  \eqref{energyest3}.
\end{proof}

\subsubsection{The first attempt for uniform bounds and the dissipation inequality} 
We want to prove
\begin{prop}\label{energyprop}
Let $\mu>\frac12 \lam$ and $(\rho,u)$ be a smooth global solution of $\mbox{(CNS)}$ with $0\leq\rho\leq \bar\rho.$  Then $u\in L^\infty((0,+\infty);L^4\cap H^1 )\cap L^2((0,+\infty); \dot{H}^1\cap \dot{H}^2),$ $\mathfrak{a} \in L^\infty((0,+\infty);H^1)\cap L^2((0,+\infty); L^6),$ and $u\cdot\na u\in L^2((0,+\infty); L^2).$ Furthermore, the following inequality holds
\begin{equation}\label{energyest1}
\begin{split}
\f{d}{dt}&\Big[ A_1\|\rho^{\f14} u\|_{L^4}^4+A_2\Big( \mu\|\na u\|_{L^2}^2+(\lam+\mu)\|\div u\|_{L^2}^2-(\mathfrak{a} , \div u)- \int f(\rho)dx\Big)+A_3\|\mathfrak{a} \|_{L^6}^2\\
&+A_4\big(\int H(\rho|1)\,dx+\|\sqrt{\rho} u\|_{L^2}^2\big) \Big]+ A_5\Big(\| |u| |\na u|\|_{L^2}^2+\|\na u\|_{L^2}^2+ \|\sqrt{\rho}u_t\|_{L^2}^2 \\
&\quad +\|\Delta \cP u\|_{L^2}^2+\|\div u-\f1{\lam+2\mu}\mathfrak{a}\|_{\dot{H}^1}^2+\|\mathfrak{a} \|_{L^6}^2+\|\na u\|_{L^6}^2\Big) \leq0,
\end{split}
\end{equation}
where $A_i(i=1,\cdots,5)$ are positive constants depending on $\mu,$ $\lam,$ and $\bar{\rho}$ and $f(\rho)$ is defined in \eqref{deffa} verifying $|f(\rho)|\lesssim H(\rho|1)$.
\end{prop}
\begin{rmk} Thanks to the energy identity and the fact that $\int f(\rho)dx\lesssim \int |\rho-1|^2\,dx$, choose $A_4$ large enough and then we can derive that
\beno
&&A_1\|\rho^{\f14} u\|_{L^4}^4+A_2\Big( \mu\|\na u\|_{L^2}^2+(\lam+\mu)\|\div u\|_{L^2}^2-( \mathfrak{a} , \div u)- \int f(\rho)dx\Big)+A_3\|\mathfrak{a} \|_{L^6}^2\\
&&+A_4\big(\int H(\rho|1)\,dx+\|\sqrt{\rho} u\|_{L^2}^2\big)\sim \|\rho^{\f14} u\|_{L^4}^4+\|\na u\|_{L^2}^2+\int H(\rho|1)\,dx+\|\sqrt{\rho} u\|_{L^2}^2+\|  \mathfrak{a} \|_{L^6}^2.
\eeno
\end{rmk}

\begin{proof} To derive the desired results, we split the proof into several steps.
	
	{\it Step 1: Estimate of $\na u$.} First, multiplying the second equation of $\mbox{(CNS)}$ with $u_t$  and taking the inner product, we get that
\begin{equation}\label{enep1}
\begin{split}
\f{d}{dt}(\f12 \mu\|\na u\|_{L^2}^2+\f12(\lam+\mu)\|\div u\|_{L^2}^2)+(\rho u_t,u_t)=-(\na \rho^\gamma,u_t)-(\rho u\cdot \na u,u_t).
\end{split}
\end{equation}

\underline{Estimate of $-(\na \rho^\gamma ,u_t)$.}
Observe that
\beno
-(\na \rho^\gamma,u_t)&=&-\f{d}{dt}(\na \rho^\gamma , u)+(\pa_t(\rho^\gamma),-\div u)=-\f{d}{dt}(\na \rho^\gamma, u)+(\gamma \rho^\gamma\div u+u\cdot \na\rho^\gamma,\div u) \\
 &\leq& -\f{d}{dt}(\na \mathfrak{a} , u)+C\|\rho\|^\gamma_{L^\infty} \|\na u\|_{L^2}^2 +(u\cdot \na \mathfrak{a} ,\div u).
\eeno
Let us focus on the last term $(u\cdot \na \mathfrak{a} ,\div u)$.  In fact, one can check that  
\beno
 \f1{\lam+2\mu}(u \mathfrak{a} ,\na \mathfrak{a} )=\pa_t \int f(\rho)\,dx,\eeno
where \ben\label{deffa}\qquad f(\rho)=
\left\{\begin{array}{l}
	\displaystyle \f1{\lam+2\mu}\big[\f{\gamma^2}{2(2\gamma-1)} (\rho-1)^2-\big(\f{\gamma-1}{2(2\gamma-1)}\rho^\gamma+ \f{\gamma(\gamma-1)}{2(2\gamma-1)}\rho\\\qquad-\f{ \gamma^2+2\gamma-1}{2(2\gamma-1)}\big)H(\rho|1)\big],\quad \textrm{when}\, \gamma>1,\\
	\displaystyle \f1{\lam+2\mu}\big[\f12 (\rho-1)^2-\big(\rho\ln\rho-\rho+1\big)\big] , \qquad \textrm{when}\quad \gamma=1.
\end{array}\right. \een
Thanks to Remark \ref{frakaequa} and $\rho\le \bar{\rho}$, we get that $|f(\rho)|\lesssim H(\rho|1)$. 

 Going back to the estimate of $-(\na \rho^\gamma ,u_t)$, we have
\beno
&&(u\cdot \na \mathfrak{a} ,\div u)=-(\mathfrak{a} \div u ,\div u)-(\mathfrak{a} u ,\na (\div u-\f1{\lam+2\mu}\mathfrak{a} ))-\f1{\lam+2\mu}(u \mathfrak{a} ,\na \mathfrak{a} ) \\&&\le C\|\mathfrak{a} \|_{L^\infty} \|\div u\|_{L^2}^2+  \|\mathfrak{a} \|_{L^\infty}^{\f13} \|\mathfrak{a} \|_{L^2}^{\f23}\|\na u\|_{L^2}  \|\na (\div u-\f1{\lam+2\mu}\mathfrak{a} )\|_{L^2}
 -\pa_t \int f(\rho)\,dx,
\eeno
which yields 
\beno  &&-(\na \mathfrak{a} ,u_t)\leq -\f{d}{dt}(\na \mathfrak{a} , u)-\pa_t \int f(\rho)\,dx+C\|\rho\|^\gamma_{L^\infty} \|\na u\|_{L^2}^2  +C\|\mathfrak{a}\|_{L^\infty} \|\div u\|_{L^2}^2\\&&\qquad\qquad\qquad+ C\|\na u\|_{L^2}  \|\na (\div u-\f1{\lam+2\mu}\mathfrak{a} )\|_{L^2}.\eeno

\underline{Estimate of $(\rho u\cdot\na u,u_t)$.}  We have
\beno
\begin{split}
|(\rho u\cdot\na u,u_t)|\leq \|\rho^{\f12}\|_{L^\infty} \|u\cdot\na u\|_{L^2} \|\rho^{\f12} u_t\|_{L^2}.
\end{split}
\eeno
Plugging these two estimates into \eqref{enep1},  we obtain that
\begin{equation}\label{enep2}
\begin{split}
\f{d}{dt}&\bigg(\f12 \mu\|\na u\|_{L^2}^2+\f12(\lam+\mu)\|\div u\|_{L^2}^2-(\mathfrak{a} ,\div u)- \int  f(\rho)\,dx\bigg)+ \|\rho^{\f12} u_t\|_{L^2}^2\\
&\leq  C\|\na u\|_{L^2}^2+C\eta \|\na (\div u-\f1{\lam+2\mu} \mathfrak{a})\|_{L^2}^2 + C  \|u\cdot\na u\|_{L^2}^2,
\end{split}
\end{equation}
where $\eta$ is a small constant and the constant $C$ depends on the initial data and $\overline{\rho}$.

{\it Step 2: Improving estimate by the elliptic system.} The second equation of $\mbox{(CNS)}$ can  rewritten into
\beno
-\mu\Delta u-(\lam+\mu)\na\div u+\na \mathfrak{a} =-\rho(u_t+u\cdot\na u).
\eeno
Set $b=\cP u=(I+\na(-\Delta)^{-1}\div)u,$ $d=\Lam^{-1}\div u,$ where $\Lam$ is a Fourier multiplier, which satisfies $\Lam^2=-\Delta.$ Then the above equations turns to
\begin{equation}
  \left\{
    \begin{aligned}
      & -\mu\Delta b=\cP\big( \rho (u_t+u\cdot\na u)\big),\\
      &-(\lam+2\mu)\Delta d-\Lam \mathfrak{a} =\Lam^{-1}\div\big( \rho (u_t+u\cdot\na u)\big).
        \end{aligned}
  \right.
  \label{eq:BD}
\end{equation}
By the standard elliptic estimate, we have
\begin{equation}\label{enep3}
\|\mu\Delta b\|_{L^2}^2+\|(\lam+2\mu)\Lam d-\mathfrak{a} \|^2_{\dot{H}^1}\leq (1+\overline{\rho})^2( \|\rho^{\f12} u_t\|_{L^2}^2+ \|u\cdot\na u\|_{L^2}^2).
\end{equation}
Combining \eqref{enep2} and \eqref{enep3}, we can get that
\begin{equation}\label{enep4}
\begin{split}
\f{d}{dt}& \Big( \mu\|\na u\|_{L^2}^2+(\lam+\mu)\|\div u\|_{L^2}^2+(\na \mathfrak{a} ,u)- \int f(\rho)\,dx\Big)\\
&+ \big(\|\rho^{\f12}  u_t\|_{L^2}^2 +\|\mu\Delta b\|_{L^2}^2+\|(\lam+2\mu)\Lam d-\mathfrak{a} \|^2_{\dot{H}^1}  \big) \leq C \|\na u\|_{L^2}^2 +C \|u\cdot\na u\|_{L^2}^2,
\end{split}
\end{equation}
where $C$ is positive constant depending on $\overline{\rho}$ and the initial data.

{\it Step 3: Estimate of $\mathfrak{a} $.}  The first equation of $\mbox{(CNS)}$ can be rewritten by
\begin{equation}\label{enep5}
\f1\gamma(\mathfrak{a} _t+u\cdot\na \mathfrak{a} ) +\f1{\lam+2\mu}\mathfrak{a} +\mathfrak{a} \div u=-(\Lam d-\f1{\lam+2\mu}\mathfrak{a} ).
\end{equation}
Then making the inner product  to the above equation with $  |\mathfrak{a} |^4 \mathfrak{a} $, we obtain that
\beno
\f16\f{d}{dt}\|\mathfrak{a} \|_{L^6}^6+\f1{\lam+2\mu}\|\mathfrak{a} \|_{L^6}^6 +\f56\int \div u |\mathfrak{a} |^6\,dx\leq \gamma\|(\Lam d-\f1{\lam+2\mu}\mathfrak{a} )\|_{L^6}\|\mathfrak{a} ^5\|_{L^{\f65}}
\eeno
which implies 
\beno
\f16\f{d}{dt}\|\mathfrak{a} \|_{L^6}^6 + \f1{\lam+2\mu} \int (1+\f56\mathfrak{a} )\mathfrak{a} ^6\,dx& \leq & C \|(\Lam d-\f1{\lam+2\mu}\mathfrak{a} )\|_{L^6}\|\mathfrak{a} ^5\|_{L^{\f65}}\\
& \leq & C \|(\Lam d-\f1{\lam+2\mu}\mathfrak{a} )\|_{L^6}\|\mathfrak{a} \|^5_{L^{6}}.
\eeno
Dividing the above estimate by $\|\mathfrak{a} \|_{L^6}^4,$ and recalling $1+\f56\mathfrak{a}\geq \f16$,   we get that
\ben\label{est:a}
\f{d}{dt}\|\mathfrak{a} \|_{L^6}^2 +  \| \mathfrak{a} \|_{L^6}^2     \leq  C \|\na (\Lam d-\f1{\lam+2\mu}\mathfrak{a} )\|^2_{L^2}  .
\een

{\it Step 4: Closing the energy estimates.}  Combining \eqref{energyest2}, \eqref{energyest3}, \eqref{enep4} and \eqref{est:a}, and choosing $\eta$ small enough, we can get
\begin{equation}\label{enep8}
\begin{split}
\f{d}{dt}&\Big[ A_1\|\rho^{\f14} u\|_{L^4}^4+A_2\Big( \mu\|\na u\|_{L^2}^2+(\lam+\mu)\|\div u\|_{L^2}^2-( \mathfrak{a} , \div u)- \int f(\rho)dx\Big)+A_3\|  \mathfrak{a} \|_{L^6}^2\\
&+A_4\big(\int H(\rho|1)\,dx+\|\sqrt{\rho} u\|_{L^2}^2\big) \Big]+ A_5\Big(\| |u| |\na u|\|_{L^2}^2+\|\na u\|_{L^2}^2 + \|\sqrt{\rho}u_t\|_{L^2}^2\\
&\quad +\|\Delta b\|_{L^2}^2+\|\Lam d-\f1{\lam+2\mu}\mathfrak{a} \|_{\dot{H}^1}^2+\| \mathfrak{a} \|_{L^6}^2\Big) \leq A_6( \| \mathfrak{a} \|_{L^6}^2 +\|\na u\|_{L^2}^2) \|\na u\|_{L^2}^2,
\end{split}
\end{equation}
where $A_i(i=1,\cdots,6)$ are positive constants depending on $\lam,$ $\mu$ and $\overline{\rho},$ and which ensure that the term $A_2(  \mathfrak{a} , \div u)$ can be controlled by $A_2(\lam+\mu)\|  \div u\|_{L^2}^2$ and $A_4\int H(\rho|1)\,dx.$ By Gronwall's inequality, the above estimate ensures that $u\in L^\infty((0,+\infty);L^4\cap H^1)\cap L^2((0,+\infty); \dot{H}^1)$, $u_t\in L^2((0,+\infty); L^2),$ $\mathfrak{a} \in L^\infty((0,+\infty); L^2)\cap L^2((0,+\infty); L^6),$ $u\cdot\na u\in L^2((0,+\infty); L^2)$. 

Using these estimates, we can improve the estimate \eqref{enep8}.  Notice that the term in the righthand side of  \eqref{enep8} can be bounded by $C\|\na u\|_{L^2}^2$. Then thanks to the energy identity \eqref{energyest2},  the dissipation inequality in the proposition is followed by the fact that for $i\ge1$ and $p=2,6$,
\ben\label{controlofu1} \|\na^i u\|_{L^p}&\le& \|\na^i \cP u\|_{L^p}+ \|\na^i \cQ   u\|_{L^p}\le \|\na^i \cP u\|_{L^p}+ \|\na^{i-1}\div u\|_{L^p}\\&  \le &\|\na^i \cP u\|_{L^p}+ \|\na^{i-1}(\div u-\f1{\lam+2\mu}\mathfrak{a} )\|_{L^p}+\f1{\lam+2\mu}\|\na^{i-1}\mathfrak{a} \|_{L^p}\notag,\\
\|\na^i u\|_{L^6}&\le& \|\na^{i+1} \cP u\|_{L^2}+ \|\na^{i}(\div u-\f1{\lam+2\mu}\mathfrak{a} )\|_{L^2}+\f1{\lam+2\mu}\|\mathfrak{a} \|_{L^6}.
\label{controlofu2}\een
\end{proof}

\subsubsection{Improving regularity estimate for $u$} In order to get the dissipation estimate for $a$, we first improve the regularity estimates for $u$ in this subsection.  We still assume that $(\rho,u)$ is a global and smooth solution of $\mbox{(CNS)}$. We set up some notations. For a function or vector field(or even a $3\times 3$ matrix) $f(t,x)$, the material derivative $\dot{f}$ is defined by
\beno
\dot{f}= f_t+u\cdot\nabla f,
\eeno
and $\div(f\otimes u)= \sum_{j=1}^3\partial_j(fu_j)$.
For two matrices $A=(a_{ij})_{3\times 3}$ and $B=(b_{ij})_{3\times 3}$, we use the notation $A:B=\sum_{i,j=1}^3a_{ij}b_{ij}$ and $AB$ is as usual the multiplication of matrix.

\begin{prop}\label{Prop41}
Let $\mu>\frac12 \lam$ and $(\rho,u)$ be a global and smooth solution of $\mbox{(CNS)}$  satisfying $0\leq \rho\leq \overline{\rho}$ and the admissible condition \eqref{admissible}. Then  there exist constants $A_i$(i=1,\dots,6) such that
\begin{equation}\label{energyest4}
\begin{split}
&\frac{d}{dt}\Big[ A_1\|\rho^{\f14} u\|_{L^4}^4+A_2\Big( \mu\|\na u\|_{L^2}^2+(\lam+\mu)\|\div u\|_{L^2}^2-( \mathfrak{a} , \div u)+\int_{\R^6}f(\rho)dx\Big)+A_3\|  \mathfrak{a} \|_{L^6}^2\\
&+A_4\big(\int H(\rho|1)\,dx+\|\sqrt{\rho} u\|_{L^2}^2\big)+A_5\|\sqrt{\rho}\dot{u}\|_{L^2}^2
\Big]+A_6\Big( \| |u| |\na u|\|_{L^2}^2+\|\na u\|_{L^2}^2 + \|\sqrt{\rho}u_t\|_{L^2}^2\\
& +\|\Delta \cP u\|_{L^2}^2+\|\div u-\f1{\lam+2\mu}\mathfrak{a} \|_{\dot{H}^1}^2+\| \mathfrak{a} \|_{L^6}^2+\|\na u\|_{L^6}^2+\|\na \dot{u}\|_{L^2}^2+\|\div u-\f{1}{2\mu+\lambda}\mathfrak{a} \|_{W^{1,6}}^2\\&+\|\na \cP u\|_{W^{1,6} }^2  \Big)\leq 0,
\end{split}
\end{equation}
where 
\beno  
&&A_1\|\rho^{\f14} u\|_{L^4}^4+A_2\Big( \mu\|\na u\|_{L^2}^2+(\lam+\mu)\|\div u\|_{L^2}^2-(\mathfrak{a} , \div u)+\int_{\R^6}f(\rho)dx\Big)+A_3\|  \mathfrak{a} \|_{L^6}^2\\
&&\quad+A_4\big(\int H(\rho|1)\,dx+\|\sqrt{\rho} u\|_{L^2}^2\big)+A_5\|\sqrt{\rho}\dot{u}\|_{L^2}^2\\
&\sim& \|\rho^{\f14} u\|_{L^4}^4+\|\na u\|_{L^2}^2+\int H(\rho|1)\,dx+\|\sqrt{\rho} u\|_{L^2}^2+\|\mathfrak{a} \|_{L^6}^2+\|\sqrt{\rho}\dot{u}\|_{L^2}^2.
\eeno
\end{prop}

\begin{proof}
We rewrite the second equation of $\mbox{(CNS)}$ as
\beno
\rho\dot{u}+\nabla  \mathfrak{a} +Lu=0.
\eeno
Then it is not difficult to check that 
\ben\label{eq41}
\begin{split}
&\rho \dot{u}_t+\rho u\cdot \nabla \dot{u}+\nabla \mathfrak{a}_{t}+\div(\nabla \mathfrak{a} \otimes u)\\
&\quad=\mu\big[\Delta u_t+\div(\Delta u\otimes u)\big]+
(\lambda+\mu)\big[\nabla\div u_{t}+\div((\nabla\div u)\otimes u)\big].
\end{split}
\een
By the  energy estimate, we derive that
\ben\label{eq42}
&&\f{d}{dt}\int \f{1}{2}\rho|\dot{u}|^2\,dx+\underbrace{-\mu\int \dot{u}\cdot\big(\Delta
u_t+\div(\Delta u\otimes u)\big)\,dx}_{\eqdefa I}\nonumber\\
&&\quad-(\lambda+\mu)\underbrace{\int \dot{u}\cdot\big((\nabla
\div u_{t})+\div((\nabla\div u)\otimes u))\big)\,dx}_{\eqdefa II}\\
&&=\underbrace{\int  \mathfrak{a}_{t}\div\dot{u}+(\dot{u}\cdot\nabla u)\cdot\nabla \mathfrak{a}\, dx}_{\eqdefa III}.\nonumber
\een
\underline{Estimate of $I$.} It is easy to check that
\beno
&&-\int \dot{u}\cdot\big(\Delta
u_t+\div(\Delta u\otimes u)\big)dx=\int \left[\nabla\dot{u}:\nabla u_t+ u\otimes\Delta u:\nabla \dot{u}\right]dx\\
&&=\int \Big[|\nabla\dot{u}|^2-\big((\nabla u\nabla u)+(u\cdot\nabla) \nabla u\big):\nabla\dot{u}-\nabla(u\cdot\nabla\dot{u}):\nabla u\Big]dx\\
&&=\int \Big[|\nabla\dot{u}|^2-(\nabla u\nabla u):\nabla\dot{u}+\big((u\cdot\nabla)\nabla\dot{u}\big):\nabla u
-(\nabla u\nabla\dot{u}):\nabla u-\big((u\cdot\nabla)\nabla\dot{u}\big):\nabla u\Big]dx\\
&&\geq \int \left[\f{3}{4}|\nabla\dot{u}|^2-C|\nabla u|^4\right]dx.
\eeno
\underline{Estimate of $II$.} Observe that
\beno
&&\div\big((\nabla\div u)\otimes u\big)=\nabla(u\cdot\nabla\div u)-\div(\div u\nabla\otimes u)+\nabla(\div u)^2,\\
&&\div\dot{u}=\div u_t+\div(u\cdot\nabla u)=\div u_t+u\cdot\nabla\div u+\nabla u:(\nabla u)^T,
\eeno
where $A^T$ means the transpose of matrix $A$. Then we get
\beno
&&-\int \dot{u}\cdot\Big[\nabla\div
u_{t}+\div\big((\nabla\div u)\otimes u\big)\Big]dx\\
&=&\int \Big[\div\dot{u}\div u_t+\div\dot{u}(u\cdot\nabla\div u) 
-\div u(\nabla\dot{u})^T:\nabla u+\div\dot{u}(\div u)^2\Big]dx\\
&=&\int \Big[|\div\dot{u}|^2-\div\dot{u}\nabla u:(\nabla u)^T-\div u(\nabla\dot{u})^T:\nabla u+\div\dot{u}(\div u)^2\Big]dx\\
&\geq&\int \Big[\f{1}{2}|\div\dot{u}|^2-\f{1}{4} |\nabla\dot{u}|^2-C|\nabla u|^4\Big]dx.
\eeno

\underline{Estimate of $III$.} We have
\beno
&&\int \mathfrak{a}_{t}\div\dot{u}+(u\cdot\nabla\dot{u})\cdot\nabla \mathfrak{a}\, dx
\\
&&=\int -\gamma\rho^{\gamma}  \div u\div\dot{u}-(u\cdot\nabla  a )\div\dot{u}+({u}\cdot\nabla\dot u)\cdot\nabla  \mathfrak{a}\, dx\\
&&=\int  -\gamma\rho^{\gamma}   \div u\div\dot{u}+\mathfrak{a}\Big[\div\big((\div\dot{u})u\big)-\div(({u}\cdot\nabla\dot u))\Big]dx\\
&&=\int  -\gamma\rho^{\gamma}  \div u\div\dot{u}+ \mathfrak{a} \Big[\div u\div\dot{u}-(\nabla u)^T:\nabla\dot{u}\Big]dx\\
&&\leq C\|\nabla u\|_{L^2  }\|\nabla\dot{u}\|_{L^2 } .
\eeno

Substituting these estimates into (\ref{eq42}) yields
\ben\label{eq43}
\begin{split}
&\f{d}{dt}\int \rho|\dot{u}|^2dx+\mu\int |\nabla\dot{u}|^2dx
+(\lambda+\mu)\int |\div\dot{u}|^2dx \\
&\leq
C\int |\nabla{u}|^4dx+C\|\nabla {u}\|^2_{L^2 }.
\end{split}
\een

To conclude the estimate by Gronwall's inequality, we will use the term $\|\sqrt{\rho}\dot{u}\|_{L^2 }$ to control $\|\nabla u\|_{L^4 }$. By Proposition \ref{energyprop} and \eqref{eq:BD}, we have
\beno
&&\|\na u\|_{L^\infty(0,\infty;L^2)}+\|\mathfrak{a}\|_{L^\infty(0,\infty;L^6)}\leq C, \\
&& \|\na b\|_{L^6}+  \|\Lambda d -\f{\mathfrak{a}}{\lambda+2\mu} \|_{L^6}\leq \|\rho\dot{u}\|_{L^2}\leq C  \|\sqrt{\rho}\dot{u}\|_{L^2},
\eeno
which together with \eqref{controlofu2} imply that
\begin{eqnarray*}
\|\nabla u\|^4_{L^4 }&\leq&\|\nabla u\|_{L^2 }\|\nabla
u\|^3_{L^6 }\leq C\|\nabla u\|_{L^6 }\|\nabla
u\|^2_{L^6 }\\
&\leq& C\|\nabla u\|^2_{L^6 }\big(\|\nabla
b\|_{L^6 }+\|\nabla d-\f{\mathfrak{a}}{\lambda+2\mu}\|_{L^6 }+\|\mathfrak{a}\|_{L^6}\big)\\
&\leq& C\|\nabla u\|^2_{L^6 }\Big(1+\|\sqrt{\rho}\dot{u}\|_{L^2 }\Big)\le C\|\nabla u\|^2_{L^6 }\Big(1+\|\sqrt{\rho}\dot{u}\|_{L^2 }^2\Big).
\end{eqnarray*}

Substituting this estimate into (\ref{eq43}) and noting that  $\|\nabla
u(t)\|^2_{L^6 }\in L^1(0,\infty)$ by Proposition \ref{energyprop}, we get by
Gronwall's inequality that
\begin{eqnarray}\label{eq:w-high}
\int \rho|\dot{u}|^2dx+\int^\infty_0\int |\nabla
\dot{u}|^2dxdt\leq C,
\end{eqnarray}
with $C$ depending only on $\overline{\rho}$ and $\rho_0,u_0$. By using this estimate, \eqref{eq43} can be improved as
\beno 
&&\f{d}{dt}\int \rho|\dot{u}|^2dx+\mu\int |\nabla\dot{u}|^2dx
+(\lambda+\mu)\int |\div\dot{u}|^2dx 
\leq
C(\|\na u\|_{L^6}^2+\|\nabla {u}\|^2_{L^2}),\eeno
from which together with  \eqref{energyest1}, \eqref{eq:BD} and Sobolev imbedding theorem will imply \eqref{energyest4}.
\end{proof}

\subsubsection{Estimate for the propagation of $\na \mathfrak{a} $}  In this subsection, we   want to give the proof to the upper bound of $\|\na u\|_{L^2((0,+\infty);L^\infty) }$ which in turn gives the estimates for propagation of $\na \mathfrak{a}$.   We want to prove:
\begin{prop}\label{prop_density}
Let $0<\alpha<1,$ $\mu>\frac12 \lam$ and $(\rho,u)$ be a global and smooth solution of $\mbox{(CNS)}$ with initial data $(\rho_0,u_0)$ verifying that $\rho_0\ge c>0,$ the admissible condition \eqref{admissible} and
\begin{equation}\label{holder}
\sup_{t\in \R^+} \|\rho(t,\cdot)\|_{C^\alpha }\leq M.
\end{equation}
Then
\ben\label{eq:blow}
 \|\mathfrak{a}\|_{L^\infty((0,+\infty);W^{1,6})\cap L^2((0,+\infty);W^{1,6})}+\|\na u\|_{L^2((0,+\infty);L^\infty) }\leq C,
\een
where $C$ depends on the initial data $(\rho_0, u_0)$ and $M$. As a consequence, there exists a constant $ \underline{\rho}=\underline{\rho}(c,M)>0$ such that for all $t\ge0$, $\rho(t,x)\ge \underline{\rho}$. Moreover, 
 \ben\label{L2naa}
	\pa_t\|\na \mathfrak{a}\|_{L^2}^2+\f{1}{4(\lambda+2\mu)}\|  \na \mathfrak{a}\|_{L^2}^2 &\leq&  
  C( \|\na \dot{u}\|_{L^2}^2+\|u\na u\|^2_{L^2}+\|\sqrt{\rho}u_t\|_{L^2}^2+\|\mathfrak{a}\|_{L^6}^2).\een 
 \end{prop}

\begin{proof} 
First, because of $x^\gamma$ is convex function when $\gamma>1$, \eqref{holder} implies that
\beno
\sup_{t\in \R^+} \|\rho^\gamma(t,\cdot)\|_{C^\alpha }\leq M.
\eeno
Next, we have the interpolation inequality
\beno
\| \na \Lambda^{-1}   \mathfrak{a}\|_{L^\infty}&\leq & 2^{\f N2} \|\mathfrak{a}\|_{L^6}+\sum_{j\geq N}2^{-j \alpha  } (2^{j \alpha  }\|\ddj  \mathfrak{a}\|_{L^\infty})\le 2^{\f N2} \|\mathfrak{a}\|_{L^6}+ 2^{-N\alpha  }  \|  \mathfrak{a}\|_{C^{\alpha}}.
\eeno
Choosing $2^{N(\f12+\alpha )}=\f{ \|  \mathfrak{a}\|_{C^{\alpha}} }{\|\mathfrak{a}\|_{L^6}}$, we get that
$
\| \na \Lambda^{-1}   \mathfrak{a}\|_{L^\infty}\leq \|\mathfrak{a}\|_{L^6}^\beta \|\mathfrak{a}\|_{C^{\alpha}}^{1-\beta} 
$
with $\beta=1-\f{1}{1+2\alpha}\in (0,1)$, which implies that
\ben\label{nau} 
\begin{split}
\|\nabla u\|_{L^\infty}^2&\leq C(\|\nabla \cP u\|^2_{L^\infty}+\| \na \Lambda^{-1}  (\div u-\f{1}{2\mu+\lambda}\mathfrak{a})\|^2_{L^\infty}+\| \na \Lambda^{-1} \mathfrak{a}\|^2_{L^\infty})\\
&\leq C(\|\nabla \cP u\|_{W^{1,6}}^2+\| (\div u-\f{1}{2\mu+\lambda}\mathfrak{a})\|^2_{W^{1,6}}+C_\eta\|\mathfrak{a}\|_{L^6}^2+\eta\|\mathfrak{a}\|_{C^{\alpha}}^{2}).
\end{split}
 \een

On the other hand, it is not difficult to derive that 
\ben\label{eq52}
&&\f{1}{\gamma}(\partial_t \nabla \mathfrak{a}+(u\cdot\nabla)\nabla \mathfrak{a})+ \f{\rho}{\lambda+2\mu} \nabla \mathfrak{a}+\f1{\gamma}\nabla u\nabla \mathfrak{a}+ \div u \nabla \mathfrak{a}  =-\rho \na(\div u -\f{1}{\lambda+2\mu}\mathfrak{a}) .
\een
 Then by energy estimates, we can derive that for $p\ge 2$,
\beno
 \f12\pa_t \|\nabla \mathfrak{a}\|_{L^p}^2  +\f{1}{(\lambda+2\mu)p}\| \nabla \mathfrak{a}  \|_{L^p}^2 &\leq& C( \|\na u\|_{L^\infty}\|\nabla \mathfrak{a}\|_{L^p}^2+\|\div u-\f{1}{2\mu+\lambda}\mathfrak{a}\|_{L^\infty}\|\na \mathfrak{a}\|_{L^p}^2
 \\&&+ \|\nabla(\div u-\f{1}{\lambda+2\mu}\mathfrak{a})\|_{L^{p}}  \|\na \mathfrak{a}\|_{L^p}),
\eeno
which implies that
\ben\label{lpa}
\begin{split}
 \pa_t \|\nabla \mathfrak{a}\|_{L^p}^2  +\f{1}{(\lambda+2\mu)p}\| \nabla \mathfrak{a}  \|_{L^p}^2 \leq& C( \|\na u\|^2_{L^\infty}\|\nabla \mathfrak{a}\|_{L^p}^2+\|\div u-\f{1}{2\mu+\lambda}\mathfrak{a}\|^2_{W^{1,6}}\|\na \mathfrak{a}\|_{L^p}^2
 \\&+ \|\nabla(\div u-\f{1}{\lambda+2\mu}\mathfrak{a})\|_{L^{p}}^2).
 \end{split}
\een
By taking $p=6$ in \eqref{lpa} and using \eqref{nau} and \eqref{energyest4}, we obtain from Gronwall's inequality that 
 $\|\mathfrak{a}\|_{L^\infty((0,+\infty);W^{1,6})\cap L^2((0,+\infty);W^{1,6})}\leq C$, from which together with \eqref{nau} implies that 
 $  \|\na u\|_{L^2((0,+\infty);L^\infty) }\leq C$.  It completes the proof to \eqref{eq:blow}.
 
  Now we go back to \eqref{lpa} with $p=2$. By Gronwall's inequality, we obtain that  $\na \mathfrak{a}\in L^\infty((0,+\infty);L^2)\cap L^2((0,+\infty);L^2)$. Thanks to the  uniform-in-time bounds obtained in the above, \eqref{lpa} with $p=2$ will yield that
\beno \pa_t \|\nabla \mathfrak{a}\|_{L^2}^2  +\f{1}{4(\lambda+2\mu)}\| \nabla \mathfrak{a}  \|_{L^2}^2 &\leq &C( \|\nabla(\div u-\f{1}{\lambda+2\mu}\mathfrak{a})\|_{L^{2}}^2+\|\div u-\f{1}{2\mu+\lambda}\mathfrak{a}\|^2_{W^{1,6}} \\
&&\quad +\|\nabla \cP u\|_{W^{1,6}}^2+C_\eta\|\mathfrak{a}\|_{L^6}^2,
\eeno
from which together with the elliptic estimate for \eqref{eq:BD}, we obtain \eqref{L2naa}.

Now using the equation of density and the estimate \eqref{eq:blow}, we have
\beno
\rho(t,x)\geq \rho_0(x)\exp^{-\int_{0}^t\|\div u\|_{L^\infty}d\tau}\ge ce^{-Ct^{\f12}}.
\eeno
 On the other hand, thanks to \eqref{afraka} and  \eqref{est:a}, we  derive that $\lim\limits_{t\rightarrow \infty} \|a(t)\|_{L^6}=\lim\limits_{t\rightarrow \infty} \|\mathfrak{a}(t)\|_{L^6}=0$, from which together with upper bounds for $\rho$ in $C^\alpha$,  we derive that 
$\lim\limits_{t\rightarrow \infty} \|a(t)\|_{L^\infty}=0$.
These two facts imply that there exists a constant $ \underline{\rho}=\underline{\rho}(c,M)>0$ such that for all $t\ge0$, $\rho(t,x)\ge \underline{\rho}$. We complete the proof to the proposition.
\end{proof}

 \subsubsection{Deriving the dissipation inequality} We want to prove 
\begin{prop}\label{energy_prop_main}
Let $0<\alpha<1,$ $\mu>\frac12 \lam,$ and $(\rho,u)$ be a global and smooth solution of $\mbox{(CNS)}$ with initial data $(\rho_0,u_0)$ verifying that $\rho_0\ge c>0,$ the admissible condition \eqref{admissible} and
$
\sup_{t\in \R^+} \|\rho(t,\cdot)\|_{C^\alpha }\leq M.
$ Then  there exist
constants $A_i(i=1,\cdots,7)$ are positive constants depending on $\mu,$ $\lam$ and $M$ such that  
\beno X(t)&\eqdefa&   A_1\|\rho^{\f14} u\|_{L^4}^4+A_2\big( \mu\|\na u\|_{L^2}^2+(\lam+\mu)\|\div u\|_{L^2}^2-(\mathfrak{a}, \div u)+\int f(\rho)dx\big)+A_3\|\mathfrak{a}\|_{L^6}^2\\
&&+A_4\big(\int H(\rho|1)\,dx+\|\sqrt{\rho} u\|_{L^2}^2\big)+A_5\|\sqrt{\rho}\dot{u}\|_{L^2}^2+A_6\|\na \mathfrak{a}\|_{L^2}^2\\
&&\sim \|    u\|_{H^1}^2+\| a\|_{H^1}^2 +\| \dot u\|_{L^2}^2, \eeno
which verifies
 \begin{equation}\label{energyest1_main}
\begin{split}
\frac{d}{dt} X(t)
+A_7\Big( \|\na^2 u\|_{L^2}^2+\|\na   u\|_{L^2}^2+\| \na a\|_{L^2}^2 +\|\na\dot u\|_{L^2}^2\Big)\leq 0.
\end{split}
\end{equation}
\end{prop}
\begin{proof} Thanks to Proposition \ref{prop_density}, we may assume that $\underline{\rho}\le \rho\le M$. From \eqref{L2naa} and \eqref{energyest4}, we get that there exist constants $A_i(i=1,\cdots,7)$ such that 
\beno 
&&\frac{d}{dt}  \big[ A_1\|\rho^{\f14} u\|_{L^4}^4+A_2\big( \mu\|\na u\|_{L^2}^2+(\lam+\mu)\|\div u\|_{L^2}^2-(\mathfrak{a}, \div u)+\int f(\rho)dx\big)+A_3\|  \mathfrak{a}\|_{L^6}^2\\
&&+A_4\big(\int H(\rho|1)\,dx+\|\sqrt{\rho} u\|_{L^2}^2\big)+A_5\|\sqrt{\rho}\dot{u}\|_{L^2}^2+A_6\|\na \mathfrak{a}\|_{L^2}^2\big]\\&&
+A_7\Big( \|\na^2 u\|_{L^2}^2+\|\na   u\|_{L^2}^2+\| \na \mathfrak{a}\|_{L^2}^2 +\|\na\dot u\|_{L^2}^2\Big)\leq 0.
\eeno	
  Thanks to the energy identity \eqref{energyest2}, the constant $A_4$ can be chosen large enough to ensure that $X(t)\ge0$. Due to the condition $\underline{\rho}\le \rho\le M$, one has 
  $\|\na a\|_{L^2}\sim \|\na \mathfrak{a}\|_{L^2}$ and $\int H(\rho|1)\,dx\sim \|\rho-1\|_{L^2}^2$, from which together with  \eqref{afraka} and $\rho \dot u+\na \mathfrak{a}=\mu\Delta u+(\lam+\mu)\na\div u$, we deduce that   $X(t)\sim \|u\|_{H^2}^2+\| a\|_{H^1}^2 +\| \dot u\|_{L^2}^2$.  It ends the proof of the proposition.
\end{proof}

  \subsection{Convergence to the equilibrium}
The aim of this subsection is to show the convergence of the solution to the equilibrium. Thanks to  Proposition \ref{prop_density}, now we may assume that $\rho\geq \underline{\rho}$.

 We begin with a crucial lemma on the estimate of the low frequency part of the solution.

\begin{lem}\label{conlf}  Let $a_0,\rho_0 u_0\in L^{p_0}(\R^3)$ with $p_0\in [1,2]$. Then if $\rho(t,x)\le M$, we have
\ben\label{conlfsol}
\begin{split}
 \int_{S(t)} \big( \gamma|\hat{a}(\xi,t)|^2+|\widehat{\rho u}(\xi,t)|^2 \big) \,d\xi \leq&  
  C(M)\big(\|a_0\|_{L^{p_0}}^2+ \|\rho_0 u_0\|_{L^{p_0}}^2\big)(1+t)^{-2\beta({p_0})}\\&+ C(M)(1+t)^{-\f32}\int_0^t \big( \|u\|_{L^2}^4+ \|a\|_{L^2}^4\big)\,ds,
  \end{split}
\een
where $S(t) =\{\xi\in \R^3: |\xi|\leq C(1+t)^{-\f12}  \}.$
\end{lem}

\begin{proof} Note that $\rho_t=a_t,$ we take the Fourier transform of $\mbox{(CNS)}$, and then multiply $\gamma\bar{\hat{a}}$ to the first equation, multiply $\overline{\widehat{\rho u}}$ to the second equation respectively to obtain that
	\beno
	\left\{
	\begin{aligned}
		& \f12\f{d}{dt} \gamma|\hat{a}|^2 + i\gamma \xi \cdot \widehat{\rho u} \bar{\hat{a}}=0,\\
		&\f12\f{d}{dt} |\widehat{\rho u}|^2+\big(\widehat{\div(\rho u\otimes u)}-\mu\widehat{\Delta u}-(\lam+\mu)\widehat{\na\div u} \big) \cdot \overline{\widehat{\rho u}}+ i  \xi\big((\gamma-1)\widehat{H(\rho|1)}+\gamma \hat{a} \big) \cdot \overline{\widehat{\rho u}}=0,
	\end{aligned}
	\right.
	\eeno
	which implies that
	\beno
	\f12 \f{d}{dt} \big(\gamma|\hat{a}|^2+  |\widehat{\rho u}|^2\big)=Re\big[ -\widehat{\div(\rho u\otimes u)}+\mu\widehat{\Delta u}+(\lam+\mu)\widehat{\na\div u}+ i  (\gamma-1)\xi \widehat{H(\rho|1)}  \big]\cdot\overline{\widehat{\rho u}}  \overset{\text{def}} {=}F(\xi,t).
	\eeno
	Integrating the above equation with time $t$, we get that
	\beno
	\gamma|\hat{a}(\xi,t)|^2+|\widehat{\rho u}(\xi,t)|^2=\gamma|\hat{a}(\xi,0)|^2+|\widehat{\rho u}(\xi,0)|^2 +2\int_0^t F(\xi, s)\,ds.
	\eeno
	Let $S(t)  \overset{\text{def}}{=}\{\xi: |\xi|\leq C(1+t)^{-\f12}  \},$ then we can split the phase space $\R^3$ into two time-dependent regions, $S(t)$ and $S(t)^c.$ Integrating the above equation over $S(t),$ and noting that $\widehat{\rho u}=\hat{u}+\widehat{au},$ and
	\beno
	\widehat{\Delta u}\bar{\hat{u}}=-|\xi|^2 |\hat{u}|^2, \quad \widehat{\na\div u} \bar{\hat{u}}= -|\xi \cdot \hat{u}|^2,
	\eeno
	we can obtain that
	\ben\label{decayeq1}
	\begin{split}
	 &\int_{S(t)} \big(\gamma |\hat{a}(\xi,t)|^2+|\widehat{\rho u}(\xi,t)|^2 \big) \,d\xi +\int_0^t \int_{S(t)} \big( \mu |\xi|^2 |\hat{u}|^2 + (\lam+\mu) |\xi\cdot \hat{u}|^2 \big) \,d\xi ds\\
	&=\int_{S(t)} \big( \gamma|\hat{a}(\xi,0)|^2+|\widehat{\rho u}(\xi,0)|^2 \big) \,d\xi + Re \int_0^t \int_{S(t)} \Big[ -\widehat{\div(\rho u\otimes u)} \cdot\overline{\widehat{\rho u}} \\
	&\qquad+\Big(\mu\widehat{\Delta u}  +(\lam+\mu)\widehat{\na\div u}\Big) \cdot\overline{\widehat{a u}} +   i  (\gamma-1)\xi \widehat{H(\rho|1)}\cdot  \overline{\widehat{\rho u}}\Big]  \,d\xi ds \\
	&\eqdefa \int_{S(t)} \big(\gamma |\hat{a}(\xi,0)|^2+|\widehat{\rho u}(\xi,0)|^2 \big) \,d\xi + B_1+B_2+B_3.
	\end{split}
\een
	From Lemma \ref{energylem1}, we have that $a,$ $u$ and $\rho u$ all belong to $L^\infty((0,+\infty);L^2),$ which means $\widehat{\rho u\otimes u}$ and $\widehat{a u}$ belong to $L^\infty((0,+\infty);L^\infty).$ Thanks to these facts, we can give estimates to the terms $B_i(i=1,2,3)$. We first have
	\ben\label{decayeq2}
	\begin{split}
	|B_1|&\leq \Big| \int_0^t \int_{S(t)}  \widehat{\div(\rho u\otimes u)}\cdot (\overline{\hat{u}}+\overline{\widehat{a u}})\,d\xi ds\Big|\\
	&\leq \eta \int_0^t \int_{S(t)}  \mu |\xi|^2 |\hat{u}|^2 \,d\xi ds  +C_{\eta} \int_0^t \int_{S(t)} \big| \widehat{\rho u\otimes u} \big|^2 \,d\xi ds +\int_0^t \int_{S(t)} |\xi| \big| \widehat{\rho u\otimes u} \big| |\widehat{a u}| \,d\xi ds\\
	 &\leq \eta \int_0^t \int_{S(t)}  \mu |\xi|^2 |\hat{u}|^2 \,d\xi ds +C_{\eta} \int_0^t \|\widehat{\rho u\otimes u}\|_{L^\infty}^2 \int_{S(t)}\,d\xi ds \\
	 &\qquad + C(1+t)^{-\f12}\int_0^t \|\widehat{\rho u\otimes u}\|_{L^\infty} \|\widehat{a u}\|_{L^\infty}\int_{S(t)}\,d\xi ds \\
	 &\leq \eta \int_0^t \int_{S(t)}  \mu |\xi|^2 |\hat{u}|^2 \,d\xi ds+C_{\eta}(1+t)^{-\f32}\int_0^t \|u\|_{L^2}^4 ds + C(1+t)^{-2}\int_0^t \|u\|_{L^2}^3 \|a\|_{L^2}\,ds.
	\end{split}
\een
	Similarly, one has
	\begin{equation}\label{decayeq3}
	\begin{split}
	|B_2|&\leq \eta \int_0^t \int_{S(t)}  \mu |\xi|^2 |\hat{u}|^2 \,d\xi ds +C_{\eta} \int_0^t \int_{S(t)} |\xi|^2 \big| \widehat{a u} \big|^2 \,d\xi ds \\
	&\leq \eta \int_0^t \int_{S(t)}  \mu |\xi|^2 |\hat{u}|^2 \,d\xi ds +C_{\eta}(1+t)^{-\f52}\int_0^t \|u\|_{L^2}^2 \|a\|_{L^2}^2\,ds,
	\end{split}
	\end{equation}
	and
	\beno
	|B_3|&\leq &\eta \int_0^t \int_{S(t)}  \mu |\xi|^2 |\hat{u}|^2 \,d\xi ds +C_{\eta} \int_0^t \int_{S(t)} |\xi|^2 \big| \widehat{a u} \big|^2 \,d\xi ds  + C\int_0^t \int_{S(t)} | \widehat{H(\rho|1)}|^2  \,d\xi ds \\
	&\leq& \eta \int_0^t \int_{S(t)}  \mu |\xi|^2 |\hat{u}|^2 \,d\xi ds +C_{\eta}(1+t)^{-\f52}\int_0^t \|u\|_{L^2}^2 \|a\|_{L^2}^2\,ds +(1+t)^{-\f32} \int_0^t   \| H(\rho|1) \|_{L^1}^2  \, ds \\
&\leq& \eta \int_0^t \int_{S(t)}  \mu |\xi|^2 |\hat{u}|^2 \,d\xi ds +C_{\eta}(1+t)^{-\f52}\int_0^t \|u\|_{L^2}^2 \|a\|_{L^2}^2\,ds +(1+t)^{-\f32} \int_0^t   \| a \|_{L^2}^4  \,  ds 
	\eeno

	Note that $a_0$ and $\rho_0 u_0$ belong to $L^{p_0}(\R^3)$ for $1\leq {p_0}<\f32$. Then for $\f1{p_0}+\f1{p_0'}=1,$ one has
	\begin{equation}\label{decayeq4}
	\begin{split}
	\int_{S(t)} \big( |\hat{a}(\xi,0)|^2+|\widehat{\rho u}(\xi,0)|^2 \big) \,d\xi&\leq  \big(\|\hat{a_0}\|_{L^{p_0'}}^2+ \|\widehat{\rho_0 u_0}\|_{L^{p_0'}}^2\big) \big( \int_{S(t)}  \,d\xi\big)^{1-\f2{p_0'}}\\
	&\leq C\big(\|a_0\|_{L^{p_0}}^2+ \|\rho_0 u_0\|_{L^{p_0}}^2\big)(1+t)^{-2\beta({p_0})}.
	\end{split}
	\end{equation}
	Plugging \eqref{decayeq2}, \eqref{decayeq3} and \eqref{decayeq4} into \eqref{decayeq1}, and choosing $\eta$ small enough, we arrive at  
	\begin{equation*} 
	\begin{split}
	&\int_{S(t)} \big( |\hat{a}(\xi,t)|^2+|\widehat{\rho u}(\xi,t)|^2 \big) \,d\xi +C\int_0^t \int_{S(t)} \big( \mu |\xi|^2 |\hat{u}|^2 + (\lam+\mu) |\xi\cdot \hat{u}|^2 \big) \,d\xi ds\\
	&\leq C\big(\|a_0\|_{L^{p_0}}^2+ \|\rho_0 u_0\|_{L^{p_0}}^2\big)(1+t)^{-2\beta({p_0})}+ C(1+t)^{-\f32}\int_0^t \big( \|u\|_{L^2}^4+ \|a\|_{L^2}^4\big)\,ds.
	\end{split}
	\end{equation*}
It ends the proof to the lemma.
\end{proof}
Now we are in a position to prove 

\begin{prop}\label{decayprop}  Let $0<\alpha<1,$ $\mu>\frac12 \lam,$ and $(\rho,u)$ be a global and smooth solution of $\mbox{(CNS)}$ with initial data $(\rho_0,u_0)$ verifying that $\rho_0\ge c>0,$ the admissible condition \eqref{admissible} and
	$
	\sup_{t\in \R^+} \|\rho(t,\cdot)\|_{C^\alpha }\leq M.
	$
 Suppose that $a_0 \in L^{p_0}(\R^3)\cap H^1(\R^3)$ and $u_0\in L^{p_0}(\R^3)\cap H^2(\R^3)$ with $ p_0\in [1,2]$. Then we have
\begin{equation}\label{decayest2}
\|u(t)\|_{H^1} + \|a(t)\|_{H^1}\leq \bar{C}(1+t)^{-\beta({p_0})},
\end{equation}
where $\beta({p_0})=\f34(\f2{p_0}-1)$ and the constant $\bar{C}$ depends only on $\underline{\rho},$ $\mu,$ $\lam,$ $M,$ $\|a_0\|_{L^{p_0}\cap H^1},$ and $\|u_0\|_{L^{p_0}\cap H^2}.$
\end{prop}

\begin{proof} We separate the proof into several steps.
	
	{\it Step 1: The first sight of the convergence.} Thanks to \eqref{conlfsol} and the fact that $a$ and $u$ belong to $L^\infty((0,+\infty);L^2),$ we have
\begin{equation}\label{decayeq6}
\begin{split}
&\int_{S(t)} \big( |\hat{a}(\xi,t)|^2+|\widehat{\rho u}(\xi,t)|^2 \big) \,d\xi \\
&\leq C\big(\|a_0\|_{L^{p_0}}^2+ \|\rho_0 u_0\|_{L^{p_0}}^2\big)(1+t)^{-2\beta({p_0})} +C \big(\|u\|_{L^\infty(L^2)}^4+\|a\|_{L^\infty(L^2)}^4\big) (1+t)^{-\f12}\\
&\leq C(1+t)^{-r_m},
\end{split}
\end{equation}
where $r_m=\min\{2\beta({p_0}),\f12\}$. Due to the fact $u=\rho u- a u$, we have
\beno
\int_{S(t)}  |\widehat{  u}(\xi,t)|^2  \,d\xi &\leq& \int_{S(t)}   |\widehat{\rho u}(\xi,t)|^2  \,d\xi+\int_{S(t)} |\widehat{a u}(\xi,t)|^2   \,d\xi\\
&\leq&C(1+t)^{-r_m}+C(1+t)^{-\f32}|\widehat{a u}(\xi,t)|_{L^\infty}^2\\
&\leq&C(1+t)^{-r_m}.
\eeno

Next, because of $\rho \dot u=\mu\Delta u+(\lam+\mu)\na\div u-\na \mathfrak{a}$, following the same argument, we can obtain
\beno
\int_{S(t)}  |\widehat{  \rho \dot {u} }(\xi,t)|^2  \,d\xi \leq  \int_{S(t)}  |\mu\widehat{ \Delta u}+(\lam+\mu)\widehat{\na\div u}-  \big((\gamma-1)\widehat{\na H(\rho|1)}+\gamma \widehat{\na a} \big)(\xi,t)|^2  \,d\xi  \leq  C(1+t)^{-1-r_m},
\eeno
which implies that
$
\int_{S(t)}  |\widehat{   \dot {u} }(\xi,t)|^2  \,d\xi   \leq  C(1+t)^{-1-r_m}.
$

 We recall the dissipation inequality \eqref{energyest1_main}. Then by frequency splitting method, it is not difficult to derive that  
 \beno
\f{d}{dt} X(t)+ \f{K}{1+t} X(t)&\leq& \f{1}{1+t}\int_{S(t)}\big( |\hat{a}(\xi,t)|^2+|\widehat{  u}(\xi,t)|^2+|\widehat{ \dot{u}}(\xi,t)|^2 \big) \,d\xi\\&\le&
  C(1+t)^{-1-r_m},
\eeno
which implies
\begin{equation}\label{informdecay}
X(t)\leq C(1+t)^{-r_m}.
\end{equation}
In particular, we have 
\begin{equation}\label{decayeq10}
\|u\|_{L^2}+ \|a\|_{L^2} \leq 2C(1+t)^{-r_m/2}.
\end{equation}

{\it Step 2: Improving the decay estimate (I).} 
We want to improve the decay estimate if $\beta(p_0)>\f14$. By definition, $r_m=\f12$. Thanks to \eqref{conlfsol} and \eqref{decayeq10}, we improve the estimate for the low frequency part as follows
\beno
\int_{S(t)} \big( |\hat{a}(\xi,t)|^2+|\widehat{\rho u}(\xi,t)|^2 \big) \,d\xi &\leq&  C(1+t)^{-2\beta({p_0})} + C(1+t)^{-\f32}\log(1+t).
\eeno
Now following the simiar argument used in the previous step, we conclude that
\beno \int_{S(t)} \big( |\hat{a}(\xi,t)|^2+|\widehat{u}(\xi,t)|^2+|\widehat{\dot{u}}(\xi,t)|^2 \big) \,d\xi &\leq&  C(1+t)^{-2\beta({p_0})} + C(1+t)^{-\f32}\log(1+t), \eeno
which implies that
\beno
\f{d}{dt} X(t)+ \f{K}{1+t} X(t)\leq C(1+t)^{-1}\big((1+t)^{-2\beta({p_0})} + (1+t)^{-\f32}\log(1+t)\big).
\eeno
We obtain that 
\begin{equation}\label{decayeq11}
X(t)\leq C\min\{(1+t)^{-2\beta({p_0})}, (1+t)^{-\f32}\log(1+t) \} .
\end{equation}
In particular, $\|u\|_{L^2}+ \|a\|_{L^2} \leq \min\{(1+t)^{-\beta({p_0})}, (1+t)^{-\f34}\log^{\f12}(1+t) \}.$

{\it Step 3:  Improving the decay estimate (II). } Finally we deal with the case that $\beta(p_0)>\f12$. By \eqref{decayeq11}, we have 
$\|u\|_{L^2}+ \|a\|_{L^2} \leq C(1+t)^{-\f12}.$ Now
we may repeat the same process in the above to get that
\beno \int_{S(t)} \big( |\hat{a}(\xi,t)|^2+|\widehat{u}(\xi,t)|^2+|\widehat{\dot{u}}(\xi,t)|^2 \big) \,d\xi &\leq&  C(1+t)^{-2\beta({p_0})} + C(1+t)^{-\f32}, \eeno
which implies that
\beno
\f{d}{dt} X(t)+ \f{K}{1+t} X(t)\leq C(1+t)^{-1}(1+t)^{-2\beta({p_0})}.
\eeno
It is enough to derive \eqref{decayest2}. We ends the proof to the proposition.
\end{proof}

\subsection{Proof of Theorem \ref{thm:main1}} Before giving the proof to Theorem \ref{thm:main1}, we first show the propagation of the   regularity for $\na^2a$.
\begin{prop}\label{energyprop1}
Let $(a,u)$ is a solution of $\mbox{(CNS)}$ with initial data $(a_0, u_0).$ Then under the assumptions of  Proposition \ref{decayprop}, it hold $a\in L^\infty((0,+\infty);H^2)\cap L^2((0,+\infty); H^2)$ and  $\na u\in L^2((0,\infty);H^2)$. 
\end{prop}
\begin{proof} Due to Proposition \ref{energy_prop_main}, we have $\mathfrak{a}\in L^\infty((0,+\infty);H^1)\cap L^2((0,+\infty); \dot{H}^1).$ Thanks to the lower and upper bounds for the density $\rho$, it is not difficult to check that $\|a\|_{H^2}\sim \|\mathfrak{a}\|_{H^2}$. Then the desired result is reduced to the proof of the propagation of $\na^2 \mathfrak{a}$. 
	
We first notice that  by Proposition \ref{prop_density}, $\na \mathfrak{a}\in L^\infty((0,\infty); L^p)$ with $p\in[2,6]$, which will be used frequently in what follows. Recall that 
\begin{equation}\label{aeq}
\f{1}{\gamma}(\mathfrak{a}_t + u\cdot\na \mathfrak{a} )+\mathfrak{a} (\div u-\f1{\mu+2\lam} \mathfrak{a}) +\f1{\mu+2\lam} \mathfrak{a}^2+ (\div u-\f1{\mu+2\lam} \mathfrak{a})+\f1{\mu+2\lam} \mathfrak{a}=0.
\end{equation}
Then it is not difficult to derive that
\beno
\begin{split}
\f1{2}\f{d}{dt} \|\na^2 \mathfrak{a}\|_{L^2}^2 &+ \f1{\mu+2\lam} \|\na^2 \mathfrak{a}\|_{L^2}^2 \leq \Big|\Big( \na^2 (\div u-\f1{\mu+2\lam} \mathfrak{a}), \na^2 \mathfrak{a}\Big) \Big| + \f1{\mu+2\lam}\Big|\Big( \na \mathfrak{a} \na \mathfrak{a}, \na^2 \mathfrak{a}\Big) \Big|\\
&+\Big|\Big( \na^2 [\mathfrak{a}(\div u-\f1{\mu+2\lam} \mathfrak{a})], \na^2 \mathfrak{a}\Big) \Big|+ \Big|\Big( \na^2 (u\cdot \na \mathfrak{a}), \na^2 \mathfrak{a}\Big) \Big|\\
&\overset{\text{def}}{=} D_1+D_2+D_3+D_4.
\end{split}
\eeno
By Cauchy-Schwartz inequality and Proposition \ref{prop_density}, we can estimate $D_i(i=1,2)$ easily by
\beno
D_1&\leq& \|\div u-\f1{\mu+2\lam} \mathfrak{a} \|_{\dot{H}^2} \|\na^2 \mathfrak{a}\|_{L^2} \leq  C_\eta \|\div u-\f1{\mu+2\lam} \mathfrak{a} \|_{\dot{H}^2}^2 +\eta \|\na^2 \mathfrak{a}\|_{L^2}^2,\\
D_2&\leq& C  \|\na \mathfrak{a}\|_{L^2}^{\f12} \|\na \mathfrak{a}\|_{L^6}^{\f32} \|\na^2 \mathfrak{a}\|_{L^2}\leq  \eta \|\na^2 \mathfrak{a}\|_{L^2}^2 + C_\eta \|\na \mathfrak{a}\|_{L^6}^2.
\eeno
For $D_3$, we have
\beno
D_3&\leq& \|\mathfrak{a}\|_{L^\infty} \|\div u-\f1{\mu+2\lam} \mathfrak{a} \|_{\dot{H}^2} \|\na^2 \mathfrak{a}\|_{L^2}+\|\na \mathfrak{a}\|_{L^3} \|\na(\div u-\f1{\mu+2\lam} \mathfrak{a}) \|_{\dot{H}^1} \|\na^2 \mathfrak{a}\|_{L^2}\\
&& +\|\div u-\f1{\mu+2\lam} \mathfrak{a} \|_{L^\infty} \|\na^2 \mathfrak{a}\|_{L^2}^2\\
&\leq &\eta \|\na^2 \mathfrak{a}\|_{L^2}^2 +C (\|\mathfrak{a}\|_{W^{1,6}}^2 +\|\div u-\f1{\mu+2\lam} \mathfrak{a} \|_{W^{1,6}}^2) \|\na^2 \mathfrak{a}\|_{L^2}^2 + C_\eta  \|\div u-\f1{\mu+2\lam} \mathfrak{a} \|_{\dot{H}^2}^2.
\eeno
For $D_4$,  thanks to integration by parts,  we obtain that
\beno
(u\cdot\na \na^2 \mathfrak{a}, \na^2 \mathfrak{a})=-\f12\big((\div u) \na^2 \mathfrak{a}, \na^2 \mathfrak{a}\big)
\eeno
which implies that
\beno
D_4&\leq& \|\na \mathfrak{a}\|_{L^3} \|\na^2 u \|_{L^6} \|\na^2 \mathfrak{a} \|_{L^2}+\|\na u\|_{L^\infty} \|\na^2 \mathfrak{a} \|_{L^2}^2+\|\div u\|_{L^\infty} \|\na^2 \mathfrak{a}\|_{L^2}^2\\
&\leq &\eta \|\na^2 \mathfrak{a}\|_{L^2}^2 +C_\eta ( \|\na \mathfrak{a}\|_{L^3}^2 \|\na^2 u \|_{L^6}^2+\|\na u\|_{L^\infty}^2 \|\na^2 \mathfrak{a} \|_{L^2}^2).
\eeno

\medskip

  Combining all the estimates in the above, we obtain that
\begin{equation}\label{aest1}
\begin{split}
&\f12\f{d}{dt} \|\na^2 \mathfrak{a}\|_{L^2}^2 + \f34 \f1{\mu+2\lam} \|\na^2 \mathfrak{a}\|_{L^2}^2 
\leq C\Big( \|\div u-\f1{\mu+2\lam} \mathfrak{a} \|_{\dot{H}^2}^2+   \|\na^2 u \|_{L^6}^2+\|\na \mathfrak{a}\|_{L^6}^2 \Big)\\&\qquad+ C\big( \| \div u-\f1{\mu+2\lam} \mathfrak{a}  \|_{W^{1,6}}^2 + \|\mathfrak{a}\|_{W^{1,6}}^2+\|\na u\|_{L^\infty}^2 \big)\|\na^2 \mathfrak{a}\|_{L^2}^2.
\end{split}
\end{equation}
 Thanks to \eqref{eq:BD}, we have 
\ben\label{uh3}
\begin{split}
&\|\na \cP u\|_{ \dot{H}^2}+\| \div u-\f1{\mu+2\lam} \mathfrak{a}  \|_{ \dot{H}^2}  \leq \|\rho \dot{u}\|_{ \dot{H}^1}\leq C ( \|\na a \dot{u}\|_{ L^2}+\| \na \dot{u}\|_{ L^2}) \\
&\leq C ( \|\na a \|_{L^3}\|  \dot{u}\|_{ L^6}+\| \na \dot{u}\|_{ L^2})
 \leq  C ( \|\na  \mathfrak{a} \|_{L^3}\|  \dot{u}\|_{ L^6}+\| \na \dot{u}\|_{ L^2}) 
 \leq C \| \na \dot{u}\|_{ L^2},\end{split}\een 
from which together with \eqref{controlofu2} imply that
$$\int_0^\infty \Big( \|\div u-\f1{\mu+2\lam}  \mathfrak{a} \|_{\dot{H}^2}^2+   \|\na^2 u \|_{L^6}^2+\|\na  \mathfrak{a}\|_{L^6}^2 \Big)  dt\le C.$$

By Gronwall's inequality, we have $ \mathfrak{a}\in L^\infty((0,+\infty); \dot{H}^2)\cap L^2((0,+\infty); \dot{H}^2)$, from which together with \eqref{uh3}, we deduce that $\na u\in L^2((0,\infty);H^2)$.  We ends the proof of the proposition.
\end{proof}

Finally, we are in the position to prove Theorem \ref{thm:main1}. 
\begin{proof}[Proof of Theorem \ref{thm:main1}] We first note that the first three results of the theorem are proved by Proposition  \ref{prop_density}-Proposition \ref{energyprop1}.
	
Finally let us give the proof to the fourth result in the theorem. For $q\in [2,4]$, one has
\beno
\|f\|_{\dot{B}^{-1+\f3q}_{q,1}}\leq \|f\|_{\dot{B}^{\f12}_{2,1}} \leq \|f\|_{L^2}^{\f12} \|\na f\|_{L^2}^{\f12}, \qquad \|f\|_{\dot{B}^{\f3q}_{q,1}}\leq \|f\|_{\dot{B}^{\f32}_{2,1}}\leq \|\na f\|_{L^2}^{\f12}\|\na^2 f\|_{L^2}^{\f12},
\eeno
from which together with \eqref{decayest2},  we can easily get \eqref{decayest1} and \eqref{decayest4}. On the other hand, these two estimates imply that there exists a time $t_0,$ such that
\beno
\|u(t_0)\|_{\dot{B}^{-1+\f3q}_{q,1}} +\|a(t_0)\|_{\dot{B}^{-1+\f3q}_{q,1}}+\|a(t_0)\|_{\dot{B}^{\f3q}_{q,1}}\leq \eta,
\eeno
where $\eta$ is sufficiently small. Then by the global well-posedness for $\mbox{(CNS)}$ in \cite{Charve,CMZ}, we obtain that $\na u \in L^1((t_0,\infty); L^\infty)$ which yields \eqref{lip} recalling that $\na u\in L^2((0,\infty);L^\infty)$.
	\end{proof}

\setcounter{equation}{0}
  \section{Global-in-time Stability for $\mbox{(CNS)}$ system}
In this section,   we want to prove Theorem \ref{thm:main2}. The proof will fall into two steps:
\begin{enumerate}
	\item By the local well-posedness for the system $\mbox{(CNS)}$,  we can show that the perturbed solution will remain close to the reference solution for a long time if initially they are close. 
	\item With the convergence result implies that the reference solution is close to the  equilibrium after a long time, we can find a time $t_0$ such that $t_0$ is far away from the initial time, and at this moment the solution is close to the equilibrium. Then it is not difficult to prove the global existence in the perturbation framework.
	\end{enumerate}

 \subsection{Setup of the problem}
Let $(\bar{\rho}, \bar{u})$ be a global smooth solution for the $\mbox{(CNS)}$ with the initial data $(\bar{\rho_0},\bar{u_0}).$ And let $(\rho, u)$ be the solution for the $\mbox{(CNS)}$ associated the initial data $(\rho_0,u_0),$ which satisfies \eqref{assm1}.
We denote 
  $h=\rho-\bar\rho$ and $v=u-\bar u$ which satisfy that error equations as follows
  \beno
  \left\{\begin{array}{l}
\partial_t h+ \div((h+\bar\rho)v)+\div (h \bar u)=0,\\
\pa_t v +v\cdot \na v-\f{1}{\rho}(\mu\Delta v+(\lambda+\mu)\na \div v) +\gamma \rho^{\gamma-2} \na h =\f{G}{\rho},
\end{array}
\right.
  \eeno
  where
  \beno
  -G=h \bar u_t+h v\cdot\na \bar u+h \bar u\cdot\na v+h \bar u\cdot\na \bar u+\bar \rho v\cdot \na \bar u+\bar \rho \bar u\cdot \na v+\gamma(\rho^{\gamma-1}- \overline{\rho}^{\gamma-1})\na \overline{\rho}.
  \eeno

 By slight modification, we rewrite the above system as
 \beno
 (ERR)\left\{\begin{array}{l}
 \partial_t h+ (\bar u+v)\cdot \na h=-(h+\bar \rho)\div v-h\div \bar u,\\
\pa_t v +v\cdot\na v-\mu\div (\f 1\rho \na v)-(\lambda+\mu)\na (\f1\rho\div v )\eqdefa H,
\end{array}
\right.
  \eeno
where
  \beno
  H&=&-\f{1}{\rho}\big[h \bar u_t+h v\cdot\na \bar u+h \bar u\cdot\na v+h \bar u\cdot\na \bar u+\bar \rho v\cdot \na \bar u+\gamma(\rho^{\gamma-1}- \overline{\rho}^{\gamma-1})\na \overline{\rho}\big]+\mu\f{\na h}{\rho^2}\na v+(\mu+\lambda)\f{\na h}{\rho^2}\div v\\
  &&-\gamma (\rho^{\gamma-2} -\bar\rho^{\gamma-2})\na h-\gamma\bar\rho^{\gamma-2}\na h+(\f{\bar \rho \bar u}{  \rho}-\f{\bar \rho \bar u}{\bar \rho}-\mu\f{\na   \rho}{  \rho^2}+\mu\f{\na \bar \rho}{\bar \rho^2})\cdot \na v +(\f{\bar \rho \bar u }{\bar \rho}-\mu\f{\na \bar \rho}{\bar \rho^2})\cdot \na v\\
  &&+(\mu+\lambda)\f{\na \bar\rho}{\bar\rho^2}\div v+(\mu+\lambda)(\f{\na \bar\rho}{ \rho^2}-\f{\na \bar\rho}{\bar\rho^2})\div v.
  \eeno

To catch the dissipation structure of the system, we apply  operators $\cQ$ and $\cP$ to  $v$-equation of $\mbox{(ERR)}$   individually to obtain that
\ben
\begin{split}
&\pa_t \cQ v  +v\cdot \na \cQ v -\mu\div (\f 1\rho \na \cQ v )-(
 \lambda+\mu)\na (\f1\rho \div \cQ v)\label{equ:cQ}\\
&\quad= \mu\div([\cQ, \f{1}{\rho}\na ])v+(
 \lambda+\mu)\na [\cQ,  \f1\rho\div  ]  v  +[\cQ, v\cdot \na] v   +\cQ H ,\\
 &\pa_t \cP v +v\cdot \na \cP v  -\mu\div (\f 1\rho \na \cP v)\label{equ:cP} = \mu\div([\cP, \f{1}{\rho}\na ])v +[\cP, v\cdot \na] v +\cP H.
 \end{split}
 \een

Before proving the stability, we  give the estimate to the term  $H$. We have

  \begin{lem}\label{lem:H}
Let $(\bar{\rho},\bar{u})$ be the smooth solution for $\mbox{(CNS)}$ satisfying \eqref{assm}. There exists a $\eps_0$ such that for any $0<\eps\leq \eps_0$, if
  \beno
  \|h\|_{\wt L^\infty_T(\dot B^{\f12,\f3p}_{2,p})}\leq \eps^{\f12},
  \eeno
 then there holds
  \beno
  \|H \|_{ L^1_T(\dot B^{\f12, \f3p-1}_{2,p})}\leq C_1\big(1+\|h\|_{\wt  L^\infty_T(\dot B^{\f12, \f3p}_{2,p})}+\|\cQ v \|_{\wt  L^\infty_T(\dot B^{  \f12, \f3p-1}_{2,p})}+\|\cP v \|_{\wt  L^\infty_T(\dot B^{  \f3p-1}_{p,1})}\big)  \\
 \times \big(\|h\|_{L^1_T(\dot B^{\f52,\f3p}_{2,p})}+\|\cP v \|_{L^1_T(\dot B^{\f3p+1}_{p,1})}+\|\cQ v \|_{ L^1_T(\dot B^{\f32, \f3p}_{2, p})}\big) ,
    \eeno
where $C_1$ is a positive constant depending only on $\mu,$ $\lam,$ and $C$ in \eqref{assm}.
  \end{lem}
\begin{proof}
We just establish the estimate to the term $\gamma (\rho^{\gamma-2} -\bar\rho^{\gamma-2})\na h,$ all the other term can be estimated similarly by Proposition \ref{prop:pro}.

Note that $\rho^{\gamma-2} -\bar\rho^{\gamma-2} =(\gamma-2) h \int_0^1 (\theta\rho+(1-\theta)\bar\rho)^{\gamma-3}\,d\theta,$ then we have
\beno
\|\gamma (\rho^{\gamma-2} -\bar\rho^{\gamma-2})\na h \|_{L^1_T(\dot B^{\f12, \f3p-1}_{2,p})}\leq C  \|h \na h \|_{L^1_T(\dot B^{\f12, \f3p-1}_{2,p})}.
\eeno
Using Proposition \ref{prop:pro} (b) with $s=\f32,$ $t=\f12,$ $\wt s=\f12,$ $\wt t=\f32,$ $\theta=0$ yields
\beno
\begin{split}
&\sum_{2^j\leq R_0} 2^{\f12 j} \| \dot\Delta_j (h\na h)\|_{L^1_T(L^2)}\\
&\leq C\|h\|_{\wt  L^2_T(\dot B^{\f32, \f3p}_{2,p})}\|\na h\|_{\wt  L^2_T(\dot B^{\f12, \f3p-1}_{2,p})}+C\|\na h\|_{\wt  L^2_T(\dot B^{\f12, \f3p-1}_{2,p})}\|h\|_{\wt  L^2_T(\dot B^{\f32, \f3p}_{2,p})}\leq C\|h\|^2_{\wt  L^2_T(\dot B^{\f32, \f3p}_{2,p})}.
\end{split}
\eeno
On the other hand, from Proposition \ref{prop:pro} (a) with $\sigma=\f3p,$ $\tau=\f3p-1,$ it follows that
\beno
\begin{split}
\sum_{2^j> R_0} 2^{\f3p-1 j} \| \dot\Delta_j (h\na h)\|_{L^1_T(L^p)}\leq C\|h\|_{\wt  L^2_T(\dot B^{\f32, \f3p}_{2,p})}\|\na h\|_{\wt  L^2_T(\dot B^{\f12, \f3p-1}_{2,p})}\leq C\|h\|^2_{\wt  L^2_T(\dot B^{\f32, \f3p}_{2,p})}.
\end{split}
\eeno
Thus, we deduce that 
\beno
\|h \na h \|_{L^1_T(\dot B^{\f12, \f3p-1}_{2,p})}\leq C\|h\|^2_{\wt  L^2_T(\dot B^{\f32, \f3p}_{2,p})}.
\eeno
Noting that the interpolation inequality 
\beno
\|h\|^2_{\wt  L^2_T(\dot B^{\f32, \f3p}_{2,p})}\leq \|h  \|^{\f12}_{L^1_T(\dot B^{\f52, \f3p}_{2,p})} \|h  \|^{\f12}_{\wt L^\infty_T(\dot B^{\f12, \f3p}_{2,p})},
\eeno
we have
\beno
\|h \na h \|_{L^1_T(\dot B^{\f12, \f3p-1}_{2,p})}\leq C \|h  \|_{L^1_T(\dot B^{\f52, \f3p}_{2,p})} \|h  \|_{\wt L^\infty_T(\dot B^{\f12, \f3p}_{2,p})}.
\eeno
  \end{proof}
  
\subsection{Long time existence of $\mbox{(ERR)}$ } 
 We want to prove that if the initial data of  $\mbox{(ERR)}$ is small, then its associated solution will be still small during a long time interval. More precisely, we have the following proposition.
 \begin{prop}\label{prop:ST}
Let $(\bar \rho, \bar u)$ associated with initial data $(\bar \rho_0, \bar u_0)$ be a global solution of $\mbox{(CNS)}$  satisfying \eqref{assm}. Given an $\eps>0,$ if the initial data of $\mbox{(ERR)}$ are determined by the following inequality
  \ben\label{assm2}
 \|(h_0,\cQ v_0)\|^{\ell }_{\dot B^{\f12}_{2,1}}+\|\cQ v_0 \|^{H}_{\dot B^{\f3p-1}_{p,1}}+\|h_0 \|^{H }_{\dot B^{\f3p}_{p,1}}+\| \cP v_0 \|_{\dot B^{\f3p-1}_{p,1}}  \leq \eps,
  \een
  then there exists a constant $\delta$ independent of $\eps,$ such that for any $t\in [0, \delta |\ln \eps|]$, there holds
    \beno
 \|(h,\cQ v)(t)\|^{\ell }_{\dot B^{\f12}_{2,1}}+\|\cQ v (t)\|^{H}_{\dot B^{\f3p-1}_{p,1}}+\|h(t) \|^{H}_{\dot B^{\f3p}_{p,1}}+\| \cP v(t) \|_{\dot B^{\f3p-1}_{p,1}}  \leq \eps^{\f12}.
     \eeno
 \end{prop}
\begin{rmk}
The assumption \eqref{assm2} comes directly from \eqref{assm1}.
\end{rmk}
 \begin{proof}
We use the continuity argument to prove the desired result. Let $\cT$ be the maximum time such that for any $t \in [0, \cT]$, there holds
    \beno
 \|(h,\cQ v)(t)\|^{\ell }_{\dot B^{\f12}_{2,1}}+\|\cQ v(t) \|^{H}_{\dot B^{\f3p-1}_{p,1}}+\|h (t)\|^{H }_{\dot B^{\f3p}_{p,1}}+\| \cP v (t)\|_{\dot B^{\f3p-1}_{p,1}}  \leq \eps^{\f12}.
     \eeno
The existence of $\cT$ can be obtained by the local well-posedness for the system. Then the proof of Proposition \ref{prop:ST} is reduced to prove that $\cT\geq \delta |\ln \eps|$ where $\delta>0$ is a constant independent of $\eps$.

{\it Step 1: Estimates for transport equation.}  Recalling the equation of $h$:
 \beno
 \pa_t h +(v+\bar u)\cdot \na h +v\cdot \na \bar\rho+h \div \bar u +(h+\bar \rho)\div v=0,
 \eeno
acting $\ddj$ to its both side and multiplying $|\ddj h|^{p-2}\ddj h$,  we get
 \beno
 \pa_t \|\ddj h\|_{L^p}^p -\int \div (v+\bar u) |\ddj h|^p\,dx&\leq& C \int| [\ddj ,(v+\bar u)\cdot \na ] h \cdot |\ddj h|^{p-1}| \,dx\\
 &&+C\int |\ddj( (h+\bar \rho)\div v+h\div \bar u) |\cdot |\ddj h|^{p-1}|\,dx,
 \eeno
 which implies that
 \beno
 \pa_t \|\ddj h\|_{L^p}
 &\leq& C(\|\div \bar u\|_{L^\infty}+\|\div v\|_{L^\infty})\|\ddj h\|_{L^p}+C\| [\ddj ,(v+\bar u)\cdot \na ] h\|_{L^p}\\
 &&+\|\ddj( (h+\bar \rho)\div v+h\div \bar u) |\|_{L^p}.
 \eeno
 Thus, by the definition the Besov space, for any $t \in [0, \cT]$, we have
  \ben\label{est:aH}
  \begin{split}
 \|h\|^H_{\wt L_t^\infty(\dot B^{\f 3p}_{p,1})}
 \leq&  \|h_0\|^H_{ \dot B^{\f 3p}_{p,1}}+  \int^t_0 \| \div \bar u\|_{L^\infty} \|h\|^H_{ \dot B^{\f 3p}_{p,1} }\,d\tau+   \int^t_0 \|h\|^H_{ \dot B^{\f 3p}_{p,1} }\| v \|_{B^{\f52, \f 3p+1}_{2,p}}\,d\tau \\
 &  +C\int^t_0 \sum_{2^j\geq R_0}2^{\f{3j}{p}} \| [\ddj , \bar u\cdot \na ] h\|_{L^p}\,d\tau+C\int^t_0 \sum_{2^j\geq R_0}2^{\f{3j}{p}}  \| [\ddj ,v\cdot \na ] h\|_{L^p}\,d\tau \\
 &+C\int^t_0\|(h+\bar \rho)\div v+h\div \bar u\|^H_{\dot B^{\f 3p}_{p,1}}\,d\tau,\end{split}
 \een
 and
   \ben\label{est:al}
   \begin{split}
 \|h\|^{\ell}_{\wt L_t^\infty(\dot B^{\f 12}_{2,1})}
 \leq&  \|h_0\|^{\ell}_{ \dot B^{\f 12}_{2,1}}+  \int^t_0 \|\div \bar u\|_{L^\infty} \|h\|^{\ell}_{ \dot B^{\f 12}_{2,1} }\,d\tau+   \int^t_0 \|h\|^{\ell}_{ \dot B^{\f 12}_{2,1}}\|  v \|_{\dot B^{\f52,\f 3p+1}_{2, p}}\,d\tau\\
 &  +C\int^t_0 \sum_{2^j \leq R_0}2^{\f{j}2} \| [\ddj , \bar u\cdot \na ] h\|_{L^2}\,d\tau+C\int^t_0 \sum_{2^j \leq R_0}2^{\f{j}2}\| [\ddj ,v\cdot \na ] h\|_{L^2}\,d\tau\\
 &+C\int^t_0\|(h+\bar \rho)\div v+h\div \bar u\|^{\ell}_{\dot B^{\f 12}_{2,1}}\,d\tau.\end{split}
 \een

Let us give the estimates to the terms in the righthand side of \eqref{est:aH} and  \eqref{est:al}.  By Proposition \ref{prop:com}, the commutators can be estimated as follows:
 \ben\label{est:acH}
 \begin{split}
& \int^t_0 \sum_{2^j\geq R_0}2^{\f{3j}{p}} \| [\ddj , \bar u\cdot \na ] h\|_{L^p}\,d\tau +\int^t_0 \sum_{2^j\geq R_0}2^{\f{3j}{p}}  \| [\ddj ,v\cdot \na ] h\|_{L^p}\,d\tau \\
  &\leq  C  \int^t_0 \bigg(\| \cQ \bar u\|_{\dot B^{ \f52,\f 3p+1}_{2, p}}+\| \cP \bar u\|_{\dot B^{ \f 3p+1}_{  p,1}} +\|\cQ v\|_{\dot B^{ \f52,\f 3p+1}_{2, p}}  +\|\cP v\|_{\dot B^{  \f 3p+1}_{ p,1 }}\bigg )\|h\|_{ \dot B^{\f12,\f 3p}_{2,p} }\,d\tau,
  \end{split}
 \een
 and
  \ben\label{est:acl}
  \begin{split}
 &\int^t_0 \sum_{2^j\leq R_0}2^{\f{j}{2}} \| [\ddj , \bar u\cdot \na ] h\|_{L^2}\,d\tau +\int^t_0 \sum_{2^j\leq R_0}2^{\f{j}{2}}  \| [\ddj ,v\cdot \na ] h\|_{L^2} \,d\tau \\
  &\leq  C  \int^t_0 \bigg (\| \cQ \bar u\|_{\dot B^{ \f52,\f 3p+1}_{2, p}}+\| \cP \bar u\|_{\dot B^{ \f 3p+1}_{  p,1}} +\|\cQ v\|_{\dot B^{ \f52,\f 3p+1}_{2, p}}  +\|\cP v\|_{\dot B^{  \f 3p+1}_{ p,1 }} \bigg)\|h\|_{ \dot B^{\f12,\f 3p}_{2,p} }\,d\tau.
  \end{split}
   \een
 Also, by the product estimates (Proposition \ref{prop:pro}), we can get
  \ben\label{est:ap}
  \begin{split}
 &\int^t_0\|(h+\bar \rho)\div v+h\div \bar u\|^{\ell}_{\dot B^{\f 12}_{2,1}}\,d\tau+\int^t_0\|(h+\bar \rho)\div v+h\div \bar u\|^H_{\dot B^{\f 3p}_{p,1}}\,d\tau \\
 &\leq C \int^t_0  \bigg (\| \cQ \bar u\|_{\dot B^{ \f52,\f 3p+1}_{2, p}}+\| \cP \bar u\|_{\dot B^{ \f 3p+1}_{  p,1}} +\|\cQ v\|_{\dot B^{ \f52,\f 3p+1}_{2, p}}  +\|\cP v\|_{\dot B^{  \f 3p+1}_{ p,1 }} \bigg) \|h\|_{ \dot B^{\f12,\f 3p}_{2,p} }\,d\tau \\
 &\quad+ C \int^t_0 \bigg( \|h\|_{ \dot B^{\f12,\f 3p}_{2,p} }+\|\bar \rho-1\|_{ \dot B^{\f12,\f 3p}_{2,p} }+1\bigg)\| \div v \|_{\dot B^{\f32, \f 3p}_{2,p}}\,d\tau.
 \end{split}
 \een

Plugging \eqref{est:acH}, \eqref{est:acl} and \eqref{est:ap} into \eqref{est:aH} and \eqref{est:al}, we  obtain that
 \ben\label{est:as}
 \|h\|_{\wt L_t^\infty(\dot B^{\f12, \f 3p}_{2, p})}
& \leq& \|h_0\|_{ \dot B^{\f12, \f 3p}_{2, p}}+C  \int^t_0 \bigg (\| \cQ \bar u\|_{\dot B^{ \f52,\f 3p+1}_{2, p}}+\| \cP \bar u\|_{\dot B^{ \f 3p+1}_{  p,1}} +\|\cQ v\|_{\dot B^{ \f52,\f 3p+1}_{2, p}}  +\|\cP v\|_{\dot B^{  \f 3p+1}_{ p,1 }} \bigg)\|h\|_{ \dot B^{\f12,\f 3p}_{2,p} }\,d\tau \notag\\
&&+ C \int^t_0 \bigg( \|h\|_{ \dot B^{\f12,\f 3p}_{2,p} }+\|\bar \rho-1\|_{ \dot B^{\f12,\f 3p}_{2,p} }+1\bigg)\| \cQ v \|_{\dot B^{\f52, \f 3p+1}_{2,p}}\,d\tau,  
\een
for any $t \in [0, \cT].$

 \medskip

 {\it Step 2: Estimates of momentum equation.} First, we deal with the compressible part of velocity.  Applying the operator $\ddj$ to the equation \eqref{equ:cQ} and multiplying $|\ddj \cQ v|^{p-2}\ddj \cQ v$, we get that
 \beno
 &&\f1p \f{d}{dt} \|\ddj \cQ v\|_{L^p}^p-\mu \int \div(\f1\rho \na \ddj \cQ v)|\ddj\cQ  v|^{p-2}\ddj \cQ v\,dx\\
 &&\quad -( \lambda+\mu)\int \na (\f1\rho \div \ddj \cQ v)|\ddj\cQ  v|^{p-2}\ddj \cQ v\,dx\\
 &&\quad=\int\ddj\cQ  H |\ddj \cQ v|^{p-2}\ddj \cQ v\,dx-\int  \ddj \cQ (v\cdot\na v) \cdot |\ddj \cQ v|^{p-2}\ddj \cQ v\,dx\\
 &&\quad\quad+C\int  \big\{\mu\div([\cQ\ddj, \f{1}{\rho}\na ])v+(
 \lambda+\mu)\na [\cQ\ddj,  \f1\rho \div]   v  \big\}|\ddj \cQ v|^{p-1}\,dx. 
 \eeno
By  Lemma A.5 and Lemma A.6 in \cite{Dan-CPDE}, we have
 \beno
 && \f{d}{dt} \|\ddj\cQ  v\|_{L^p}^p+c_p2^{2j} \|\ddj \cQ v\|_{L^p}^p \\
 &&\qquad\leq C  \int |\ddj \cQ H  | \cdot|\ddj v|^{p-1}\,dx + C \int | \ddj\cQ (v\cdot\na v )| \cdot |\ddj \cQ v |^{p-1} \,dx \\
 &&\qquad+C\int \Big|\mu\div([\cQ\ddj, \f{1}{\rho}\na ])v+(
 \lambda+\mu)\na [\cQ\ddj,  \f1\rho \div]   v\Big| \cdot |\ddj \cQ v|^{p-1}\,dx. 
 \eeno
 Thus, by the definition of Besov space, for any $t \in [0, \cT]$, we obtain that
  \ben\label{est:cQ1}
  \begin{split}
& \|\cQ v\|_{\wt L^\infty_t(\dot B^{\f12, \f 3p-1}_{2, p })}+\|\cQ v\|_{L^1_t(\dot B^{\f52 ,\f 3p+1}_{2, p })}\leq C \|\cQ v_0\|_{ \dot B^{\f12, \f 3p-1}_{2, p }} +C  \|v\cdot\na v\|_{L^1_t(\dot B^{\f12, \f 3p-1}_{2,p })}+C\|\cQ H\|_{L^1_t(\dot B^{\f12, \f 3p-1}_{2,p })}\\
&\qquad\quad +C\int_0^t \sum_{2^j\leq R_0}2^{\f{j}2}\| [\ddj\cQ , \f{1}{\rho}\na] v\|_{ L^2} \,d\tau+C\int_0^t \sum_{2^j\geq R_0}2^{j\f3p }\| [\ddj\cQ , \f{1}{\rho}\na] v\|_{ L^p}\,d\tau .\end{split}
 \een

 For the commutators,  applying Proposition \ref{prop:com} and Lemma \ref{Lem:com-est}  yields that
 \ben\label{est:cQ2}
 \begin{split}
 & \int_0^t\bigg(\sum_{2^j\leq R_0}2^{\f{j}2}\| [\ddj\cQ , \f{1}{\rho}\na] v\|_{ L^2} + \sum_{2^j\geq R_0}2^{\f{3j}p}\| [\ddj\cQ , \f{1}{\rho}\na] v\|_{ L^p}\bigg)\,d\tau\\
   \leq&
  C\int_0^t\bigg(\sum_{2^j\leq R_0}2^{\f{j}2} \| [\ddj\cQ, (\f{1}{\bar\rho}-1)\na] \cP v \|_{L^2}+\sum_{2^j\leq R_0}2^{\f{j}2} \| [\ddj\cQ, \f{h}{\rho\bar\rho}\na]\cP v \|_{L^2}\bigg)\,d\tau\\
  &+C\int_0^t\bigg(\sum_{2^j\geq R_0}2^{j(\f3p-1)} \| [\ddj\cQ, (\f{1}{\bar\rho}-1)\na] \cP v \|_{L^p}+\sum_{2^j\geq R_0}2^{ \f{3j}p} \| [\ddj\cQ, \f{h}{\rho\bar\rho}\na]\cP v \|_{L^p}\bigg)\,d\tau\\
&+
  C\int_0^t\bigg(\sum_{2^j\leq R_0}2^{\f{j}2} \| [\ddj\cQ, (\f{1}{\bar\rho}-1)\na]\cQ v \|_{L^2}+\sum_{2^j\leq R_0}2^{\f{j}2} \| [\ddj\cQ, \f{h}{\rho\bar\rho}\na] \cQ v \|_{L^2}\bigg)\,d\tau\\
  &+C\int_0^t\bigg(\sum_{2^j\geq R_0}2^{j(\f3p-1)} \| [\ddj\cQ, (\f{1}{\bar\rho}-1)\na] \cQ v \|_{L^p}+\sum_{2^j\geq R_0}2^{ \f{3j}p} \| [\ddj\cQ, \f{h}{\rho\bar\rho}\na] \cQ v \|_{L^p}\bigg)\,d\tau\\
  \leq & C \int_0^t\bigg(\|\bar \rho-1\|_{\dot B^{\f3p}_{p,1}\cap\dot B^{\f3p+1}_{p,1}} +\|h\|_{\dot B^{\f52,\f3p}_{2,p}}\bigg)\bigg(\|\cQ v\|_{\dot B^{\f12, \f3p-1}_{2,p}}+\|\cP v\|_{\dot B^{  \f3p-1}_{p,1}}\bigg)\,d\tau\\
  &+ \eps^\f12 \bigg(\|\cQ v\|_{ L^1_t(\dot B^{\f52, \f3p+1}_{2,p})}+\|\cP v\|_{ L^1_t(\dot B^{  \f3p+1}_{p,1})}\bigg).\end{split}
 \een
 Using Proposition \ref{prop:pro}, we can get that
 \ben\label{est:cQ3}
 \|v\cdot\na v\|_{ L^1_t(\dot B^{\f12, \f 3p-1}_{2,p })} \leq C\int_0^t \big( \|\cP v\|_{ \dot B^{\f 3p-1 }_{p,1} }+\|\cQ v\|_{ \dot B^{\f12, \f 3p-1 }_{2, p}}\big)\big( \|\cP v\|_{ \dot B^{\f 3p+1 }_{p,1} }+\|\cQ v\|_{\dot B^{\f52, \f 3p+1 }_{2, p}}\big)\,d\tau.
 \een
 
 Plugging the estimates \eqref{est:cQ2}, \eqref{est:cQ3} and the estimates of $H$ (Lemma \ref{lem:H}) into \eqref{est:cQ1}, and noting that
\beno
\|h\|_{\dot B^{\f52,\f3p}_{2,p}}\leq C\|h\|_{\dot B^{\f12,\f3p}_{2,p}},
\eeno 
  we obtain that
\ben\label{est:cQ}
\begin{split}
 \|\cQ v\|_{\wt L^\infty_t(\dot B^{\f12, \f 3p-1}_{2, p  })}+\|\cQ v\|_{L^1_t(\dot B^{\f52, \f 3p+1}_{ 2,p  })}&\leq C\|\cQ v_0\|_{\dot B^{\f12, \f3p-1}_{2,p }}+\eps^{\f12}\bigg(\|\cQ v\|_{L^1_t(\dot B^{\f52, \f 3p+1}_{ 2,p  })} +\|\cP v\|_{L^1_t(\dot B^{  \f 3p+1}_{p,1  })}\bigg)  \\
 &\quad + C\int_0^t\bigg(\|\cP v\|_{   \dot B^{ \f 3p-1}_{ p,1 } }+\|\cQ v\|_{   \dot B^{ \f12, \f 3p-1}_{ 2,p } }+\|h\|_{   \dot B^{ \f12, \f 3p}_{ 2,p } }
 \bigg)\,d\tau,
 \end{split}
 \een
for any $t \in [0, \cT].$
\bigskip

Next, we establish the estimates of incompressible part, $\cP v$. Applying the operator $\ddj$ to both sides of \eqref{equ:cP} and multiplying $|\ddj \cP v|^{p-2}\ddj \cP v$, we get that
 \beno
 &&\f1p \pa_t \|\ddj \cP v\|_{L^p}^p-\mu\int \div(\f1\rho \na \ddj \cP v)|\ddj\cP  v|^{p-2}\ddj \cP v\,dx\\
 &&\qquad=\int  \ddj\cP  H |\ddj \cP v|^{p-2}\ddj \cP v\,dx-\int  \ddj \cP (v\cdot\na v) \cdot |\ddj \cP v|^{p-2}\ddj \cP v\,dx\\
 &&\qquad\quad +\int  \mu\div([\cP\ddj, \f{1}{\rho}\na ])v  
 \cdot |\ddj \cP v|^{p-2}\ddj \cP v\,dx,
 \eeno
which implies that
 \beno
 && \pa_t \|\ddj\cP  v\|_{L^p}^p+c_p2^{2j} \|\ddj \cP v\|_{L^p}^p \\
 &&\qquad\leq C\bigg( \int |\ddj \cP H | |\ddj \cP v|^{p-1} \,dx+  \int | \ddj\cP (v\cdot\na v )| \cdot |\ddj \cP v |^{p-1} \,dx\bigg)\\
 &&\qquad+\mu\int |\div([\cP\ddj, \f{1}{\rho}\na ])v | \cdot |\ddj \cP v|^{p-1}\,dx. \eeno
 Thus, by the definition of Besov space, for any $t \in [0, \cT]$, we derive that
  \ben\label{est:cP1}
  \begin{split}
 \|\cP v\|_{\wt L^\infty_t(\dot B^{ \f 3p-1}_{ p,1 })}+\|\cP v\|_{L^1_t(\dot B^{ \f 3p+1}_{ p,1 })}
 \leq& C \|\cP v_0\|_{ \dot B^{  \f 3p-1}_{ p,1 }} +C  \|v\cdot\na v\|_{L^1_t(\dot B^{  \f 3p-1}_{ p,1 })}+C\|\cP H\|_{L^1_t(\dot B^{  \f 3p-1}_{ p,1 })}\\
& +C\int_0^t \sum_{j\in\mathbb N}2^{j\f3p}\| [\ddj\cP , \f{1}{\rho}\na] v\|_{ L^2} \,d\tau.
\end{split}
 \een
By the same argument as the proof of inequality \eqref{est:cQ2}, we have
 \ben\label{est:cP2}
 \begin{split}
 \int_0^t\bigg(\sum_{j\in\mathbb N} 2^{ \f{3j}p  }\| [\ddj\cP , \f{1}{\rho}\na] v\|_{L^p}\bigg)\,dx\leq&C\bigg(\|\cQ v\|_{  L^1_t(\dot B^{\f12, \f3p-1}_{2,p})}+\|\cP v\|_{  L^1_t(\dot B^{  \f3p-1}_{p,1})}\bigg)\\
   & +\eps^{\f12}\bigg(\|\cQ v\|_{   L^1_t(\dot B^{\f52, \f3p+1}_{2,p} )}+\|\cP v\|_{   L^1_t(\dot B^{  \f3p+1}_{p,1})}\bigg).
   \end{split}
 \een
 Plugging the estimates \eqref{est:cP2} and \eqref{est:cQ3} into \eqref{est:cP1} will imply that for any $t \in [0, \cT],$ there holds
\ben\label{est:cP}
\begin{split}
 \|\cP v\|_{\wt L^\infty_t(\dot B^{ \f 3p-1}_{ p,1 })}+\|\cP v\|_{L^1_t(\dot B^{ \f 3p+1}_{ p,1 })}
 \leq& C\|\cP v_0\|_{\dot B^{\f3p-1}_{p,1}}+ C\int_0^t\bigg(\|\cP v\|_{   \dot B^{ \f 3p-1}_{ p,1 } }+\|\cQ v\|_{   \dot B^{ \f12, \f 3p-1}_{ 2,p } }+\|h\|_{   \dot B^{ \f12, \f 3p}_{ 2,p } }
 \bigg)\,d\tau \\
 &+\eps^{\f12}\bigg(\|\cQ v\|_{L^1_t(\dot B^{\f52, \f 3p+1}_{ 2,p  })} +\|\cP v\|_{L^1_t(\dot B^{  \f 3p+1}_{p,1  })}\bigg).
 \end{split}
 \een

{\it Step 3: Closing the energy estimate.}
Combining the estimates \eqref{est:as}, \eqref{est:cQ} and \eqref{est:cP}, and choosing a suitable $\delta_1$ such that $\delta_1C \leq \f{1}{2}$,  we obtain that
\beno
\begin{split}
\delta_1 \|h\|_{\widetilde{L}^\infty_t(\dot{B}^{\f12, \f3p}_{2,p})} +\|\cQ v\|_{\widetilde{L}^\infty_t(\dot B^{\f12, \f3p-1}_{2, p })}+&\|\cP v\|_{\widetilde{L}^\infty_t(\dot B^{  \f3p-1}_{  p,1 })}\leq C\bigg( \|h_0\|_{ \dot B^{\f12, \f 3p}_{p,1}}+ \|\cQ v_0\|_{ \dot B^{\f12, \f 3p-1}_{2, p}}+ \|\cP v_0\|_{ \dot B^{\f 3p-1}_{p,1}}\bigg)\\
&\qquad\qquad+C\int^t_0 \bigg(\delta_1 \|h\|_{ \dot{B}^{\f12, \f3p}_{2,p} } +\|\cQ v\|_{ \dot B^{\f12, \f3p-1}_{2, p } }+\|\cP v\|_{ \dot B^{ \f3p-1}_{ p,1 } }
\bigg)\,d\tau,
\end{split} \eeno
for any $t \in [0, \cT].$
By the Gronwall's inequality, we get that for any $t \in [0, \cT]$, there holds
\beno
\delta_1 \|h\|_{\widetilde{L}^\infty_t(\dot{B}^{\f12, \f3p}_{2,p})} +\|\cQ v\|_{\widetilde{L}^\infty_t(\dot B^{\f12, \f3p-1}_{2, p })}+\|\cP v\|_{\widetilde{L}^\infty_t(\dot B^{  \f3p-1}_{  p,1 })}\leq  C\bigg( \|h_0\|_{ \dot B^{\f12, \f 3p}_{p,1}}+ \|\cQ v_0\|_{ \dot B^{\f12, \f 3p-1}_{2, p}}+ \|\cP v_0\|_{ \dot B^{\f 3p-1}_{p,1}}\bigg) e^t.
\eeno
According to  the definition of $\cT$, this implies that  $\cT\geq \delta |\ln \eps|$ for a suitable $\delta$ independent of $\eps.$ Then the proof is completed.
 \end{proof}

\subsection{Proof of Theorem \ref{thm:main2}.}

 Now we are in a position to prove Theorem \ref{thm:main2}. 
 \begin{proof}[Proof of Theorem \ref{thm:main2}]
 First, thanks to Theorem \ref{thm:main1}, we can choose $t_0=\f12(1+|\delta\ln\eps|)$ such that
  \beno \|(\bar{\rho}-1)(t_0)\|_{\dot{B}^{\f12,\f32}_{2,2}}+\|\bar{u}(t_0)\|_{\dot{B}^{\f12}_{2,1}}\lesssim  (1+|\delta\ln\eps|)^{-\beta(p_0)/2}. \eeno
Recall that $\rho-1=h+(\bar{\rho}-1),$ $u=v+\bar{u},$ then from Proposition \ref{prop:ST},  we derive that
\beno
\begin{split}
\|(\rho-1,\cQ u )(t_0)\|^{\ell}_{\dot B^{\f12}_{2,1}}&+\|\cQ u  (t_0)\|^{H}_{\dot B^{\f3p-1}_{p,1}}+\|(\rho-1)(t_0) \|^{H}_{\dot B^{\f3p}_{p,1}}+\|(\cP u)(t_0)\|_{\dot B^{\f3p-1}_{p,1}}\\
&\lesssim \eps^{\f12}+ (1+|\delta\ln\eps|)^{-\beta(p_0)/2}\lesssim (1+|\delta\ln\eps|)^{-\beta(p_0)/2}.
\end{split}
\eeno
This means that at $t_0$, the system $\mbox{(CNS)}$ is in the close-to-equilibrium regime. Then thanks to the results in  \cite{CMZ,MN1,MN2}, we obtain the global existence for $(\rho-1,u)$. Moreover due to the definition of $\cT$ and (\ref{decayest1}-\ref{decayest4}), we conclude that
  \beno
  \delta_1\|h\|_{\widetilde{L}^\infty_t(\dot{B}^{\f12, \f3p}_{2,p})} +\|\cQ v\|_{\widetilde{L}^\infty_t(\dot B^{\f12, \f3p-1}_{2, p })}+\|\cP v\|_{\widetilde{L}^\infty_t(\dot B^{  \f3p-1}_{  p,1 })}\lesssim \min\{(1+|\delta\ln\eps|)^{-\beta(p_0)/2}, (1+t)^{-\beta(p_0)/2}+\epsilon\}.
  \eeno
It completes the proof to  Theorem \ref{thm:main2}.
 \end{proof}

\setcounter{equation}{0}
 \section{Construction of a global solutions with a class of large initial data}
  Inspired by \cite{PZ}, in this section, we will construct a global solution for compressible Navier-Stokes equations with the vertical component of the initial data $(\cP u_0)^3$ could be arbitrarily large.  
  
  \subsection{Reduction of the problem}
Given $a_0\in \dot{B}^{\f12,\f3p-1}_{2,p}$, $\cQ u_0\in
\dot{B}^{\f12,\f3p-1}_{2,p}$ and $\cP u_0\in \dot{B}^{\f3p-1}_{p,1}$ with $\|a_0\|_{\dot{B}^{\f12,\f3p-1}_{2,p}}$ being sufficiently small, it follows by a similar argument as that in \cite{CMZ} that there exists a positive time $T$ so that $\mbox{(CNS)}$ has a unique solution $(a,u)$ with
\begin{equation}\label{local}
\begin{split}
&a\in\cC((0,T]; \dot{B}^{\f12,\f3p-1}_{2,p}),\qquad \cQ u\in L^\infty((0,T];\dot{B}^{\f12,\f3p-1}_{2,p})\cap L^1((0,T);\dot{B}^{\f52,\f3p+1}_{2,p}),\\
&\cP u\in L^\infty((0,T];\dot{B}^{\f3p-1}_{p,1})\cap L^1((0,T);\dot{B}^{\f3p+1}_{p,1}).
\end{split}
\end{equation}

Next we only need to give the {\it a priori} estimates to the solution. Observe that  the system $\mbox{(CNS)}$ can be recast as
\ben\label{neweq}
 \left\{\begin{array}{l}
\partial_t a+u\cdot \na a +\div \cQ u=-a \div  u,\\
\pa_t \cQ u +u\cdot \na \cQ u-\mu\Delta \cQ u-(\mu+\lambda)\na \div \cQ u+\gamma\na a= \cQ W_\cQ,\\
\pa_t (\cP u)^h+u\cdot \na (\cP u)^h -\Delta  (\cP u)^h = [\cP, u\cdot\na ]u^h +(\cP W_\cP)^h,\\
 \pa_t (\cP u)^3+\cP(u\cdot \na u^3 )-\Delta   (\cP u)^3 =  (\cP W_\cP)^3,
\end{array}
\right.
\een
where
 \beno
  W_\cQ&=&-\f{\na a} {(1+a)^2}\na u-\f{a\na a}{1+a}  +\mu\div ((\f1{1+a}-1)\na u)\\
 && +(\mu+\lambda)\na ((\f1{1+a}-1)\div u)+[\cQ ,u\cdot\na] u+\gamma(\rho^{\gamma-1}-1)\na a,
  \eeno
and
 \beno
 W_\cP=
-\f{\na a} {(1+a)^2}\na u   +\mu \div ((\f1{1+a}-1)\na u).
  \eeno
Now, to establish the uniform estimates for the solution, we first give estimates to $\cQ W_\cQ$ and $\cP W_\cP.$
\begin{lem}\label{lem:G1}
There exists a universal constant $C$ such that
\beno
\| \cQ W_\cQ\|_{\dot B^{\f12, \f3p}_{2,p}}\leq  C (\|\cQ u\| _{\dot B^{\f52,\f3p+1}_{2,p}}+\|\cP u\|_{\dot B^{\f3p+1}_{p,1}} +\|a\|_{\dot B^{\f52,\f3p}_{2,p}})   (\|a\|_{\dot B^{\f12,\f3p }_{2,p}}  +\|\cQ u\| _{\dot B^{\f12,\f3p-1}_{2,p}}+\|(\cP u)^h\|_{\dot B^{\f3p-1}_{p,1}}),
\eeno
and
\beno
\|\cP W_\cP\|_{ \dot B^{\f3p-1}_{p,1} }
 \leq C\|a\|_{B^{\f12,\f3p}_{2,p}}  (\|\cQ u\|_{B^{\f52,\f3p+1}_{2,p}}+\|\cP u\|_{ \dot B^{\f3p+1}_{p,1} })+C\|a\|_{B^{\f52,\f3p}_{2,p}}  (\|\cQ u\|_{B^{\f12,\f3p-1}_{2,p}}+\|\cP u\|_{ \dot B^{\f3p-1}_{p,1} }).
\eeno
\end{lem}
\begin{proof}
 For $\cQ W_\cQ$ and $\cP W_\cP,$ the most difficult term is $[\cQ, u\cdot\na ]u,$ and the others can be estimated by the Proposition \ref{prop:pro}--Propostion \ref{prop:comp} directly. So we just focus on the term $[\cQ, u\cdot\na ]u.$

Because of $u=\cQ u+\cP u,$ thus,
\beno
[\cQ, u\cdot\na ]u&=&[\cQ, (\cQ u+\cP u)\cdot\na ](\cQ u+\cP u)\\
&=&[\cQ, (\cQ u)\cdot\na ]\cQ u+[\cQ, (\cQ u)\cdot\na ]\cP u+[\cQ, (\cP u)\cdot\na ]\cQ u+[\cQ, (\cP u)\cdot\na ]\cP u.
\eeno
Here, by Proposition \ref{prop:pro}--Proposition \ref{prop:comp}, it is not difficult to get that
\beno
&&\|[\cQ, (\cQ u)\cdot\na ]\cQ u+[\cQ, (\cQ u)\cdot\na ]\cP u +[\cQ, (\cP u)\cdot\na ]\cQ u+[\cQ, (\cP u)\cdot\na ](\cP u)^h\|^{\ell}_{\dot B^{\f12}_{2,1}} \\
&\leq& C (\|\cQ u\| _{\dot B^{\f52,\f3p+1}_{2,p}}+\|\cP u\|_{\dot B^{\f3p+1}_{p,1}} )  (\|\cQ u\| _{\dot B^{\f12,\f3p-1}_{2,p}}+\|(\cP u)^h\|_{\dot B^{\f3p-1}_{p,1}}).
\eeno
For the remainder term $[\cQ, (\cP u)\cdot\na ](\cP u)^3,$ by using $\div \cP u=0$, we can obtain 
\beno
[\cQ, (\cP u)\cdot\na ](\cP u)^3= \cQ ((\cP u)\cdot\na (\cP u)^3)&=&\cQ\big((\cP u)^h\cdot\na_h (\cP u)^3-(\cP u)^3 \div_h (\cP u)^h\big)\\
&=&\cQ \big((\cP u)^h\cdot\na_h (\cP u)^3\big)-[\cQ, (\cP u)^3   ]\div_h (\cP u)^h.
\eeno
Thus, by Proposition \ref{prop:pro}--Propostion \ref{prop:comp}, we complete the proof of this lemma.
\end{proof}
 
\subsection{Proof of Theorem  \ref{thm:main3}} We finally give the proof to   Theorem  \ref{thm:main3}.

\begin{proof}[Proof of Theorem  \ref{thm:main3}]The proof of the theorem falls into four steps.
 
{\it Step 1: Estimates for the low frequency part of the solution.}  Applying $\ddj$ on the both side of the first equation of \eqref{neweq} and multiplying $ \ddj a$, we have
 \beno
 &&\f1p\f{d}{dt} \|\ddj a\|_{L^2}^2 -\int \div  \cQ u  |\ddj a|^2\,dx+\int \ddj \div \cQ u\cdot \ddj a\,dx \\
 &&\leq C \int| [\ddj ,u\cdot \na ] a \cdot  \ddj a|\,dx   +C\int |\ddj( a\div\cQ  u )  \cdot   \ddj a  |\,dx,
 \eeno
 which implies that
 \ben\label{est:alow}
 \begin{split}
 &\f12\f{d}{dt} \|\ddj a\|^2_{L^2}+\int \ddj \div \cQ u\cdot \ddj a\,dx\\
 &\leq C \|\div \cQ  u\|_{L^\infty}\|\ddj a\|^2_{L^2}+C\big(\| [\ddj ,u\cdot \na ] a\|_{L^2}+\|\ddj(a\div \cQ u )  \|_{L^2}\big) \|\ddj a\|_{L^2}.\end{split}
 \een

Next, taking $\Lambda^{-1} \div$ on the both sides of the second equation of \eqref{neweq} yields that
  \ben\label{equ:dlow}
  \pa_t d+u\cdot \na d -(2\mu+\lambda)\Delta d-\gamma\Lambda a=\Lambda^{-1} \div \cQ W_\cQ -[\Lambda^{-1} \div, u\cdot \na] \cQ u,
  \een
  where $d=\Lambda^{-1} \div \cQ u$.

Acting the operator $\ddj$ on the both side of \eqref{equ:dlow} and multiplying $ \ddj d$, we get that
 \ben\label{est:dlow}
 \begin{split}
&\f12 \f{d}{dt} \|\ddj d\|^2_{L^2}+(2\mu+\lambda)\|\na \ddj d\|_{L^2}^2-\gamma\int \ddj a\cdot   \ddj \Lambda d\,dx\\
& \leq C\|\div \cQ u\|_{L^\infty}\|\ddj d\|^2_{L^2}+\|\ddj d\|_{L^2}\|[\ddj, u\cdot \na] d\|_{L^2}\\
&\quad +\|\ddj d\|_{L^2}\big(\|\ddj \cQ W_\cQ\|_{L^2}+\|\ddj [\Lambda^{-1} \div, u\cdot \na] \cQ u\|_{L^2}\big).\end{split}
 \een

Now, we introduce a new auxiliary function $w\overset{def}{=}(2\mu+\lambda)\Lambda a-d,$ which satisfies that
\ben\label{equ:auxF}
\begin{split}
&\pa_t w+u\cdot\na w+\gamma\Lambda a\\
=&-(2\mu+\lambda)[\Lambda, u\cdot\na ] a-(2\mu+\lambda)\Lambda(a\div u)-\Lambda^{-1} \div \cQ W_\cQ +[\Lambda^{-1} \div, u\cdot \na] \cQ u.
\end{split}
\een
Applying $\ddj$ on the both side of \eqref{equ:auxF} and multiplying $ \ddj w$, we  obtain that
  \ben\label{est:auxF}
  \begin{split}
  &\f12\f{d}{dt} \|\ddj w\|_{L^2}^2+(2\mu+\lambda)\gamma\|\ddj \Lambda a\|_{L^2}^2-\gamma\int \ddj a \cdot \ddj \Lambda d \,dx \\
  &\leq C\bigg(\|\div u\|_{L^\infty}\|\ddj w\|_{L^2}^2+\|[\ddj, u\cdot \na]w\|_{L^2}\|\ddj w\|_{L^2} \bigg)\\
  &\quad+C\bigg(\|\ddj [\Lambda, u\cdot\na ] a\|_{L^2}+\|\ddj \Lambda(a\div u)\|_{L^2}+\|\ddj  \cQ W_\cQ\|_{L^2}\bigg) \|\ddj w\|_{L^2}\\
  &\quad+C\|\ddj [\Lambda^{-1} \div, u\cdot \na] \cQ u\|_{L^2}\|\ddj w\|_{L^2}.
\end{split}  \een

Putting together the estimates \eqref{est:alow}, \eqref{est:dlow} and \eqref{est:auxF}, we arrive at 
\ben\label{est:a-d-aux1}
\begin{split}
& \f{d}{dt}\bigg(\gamma\|\ddj a\|_{L^2}^2+(1-\delta)\|\ddj d\|^2_{L^2}+\delta \|\ddj w\|_{L^2}^2  \bigg)+ \|\na \ddj d\|_{L^2}^2+\delta \|\Lambda \ddj a\|_{L^2}^2   \\
&\leq C \|\div \cQ  u\|_{L^\infty}\|\ddj a\|^2_{L^2}+C\bigg(\| [\ddj ,u\cdot \na ] a\|_{L^2}+\|\ddj(a\div \cQ u ) |\|_{L^2}\bigg) \|\ddj a\|_{L^2}\\
&\quad+ C\|\ddj d\|_{L^2}\bigg(\|\div \cQ u\|_{L^\infty}\|\ddj d\|_{L^2}+\|[\ddj, u\cdot \na] d\|_{L^2}\bigg)\\
&\quad+C\bigg(\|\ddj \cQ W_\cQ\|_{L^2}+\| \ddj[\Lambda^{-1} \div, u\cdot \na] \cQ u\|_{L^2}\bigg)\|\ddj d\|_{L^2}\\
&\quad+C\delta\bigg (\|\div u\|_{L^\infty}\|\ddj w\|_{L^2}^2 +\|[\ddj,u\cdot \na]w\|_{L^2}\|\ddj w\|_{L^2} \bigg)\\
&\quad+C\delta\bigg (\|\ddj [\Lambda, u\cdot\na ] a\|_{L^2}+\|\ddj \Lambda(a\div u)\|_{L^2}+\|\ddj   \cQ W_\cQ\|_{L^2} \bigg) \|\ddj w\|_{L^2}\\
 &\quad+C\delta  \|\ddj [\Lambda^{-1} \div, u\cdot \na] \cQ u\|_{L^2}\|\ddj w\|_{L^2}  .\end{split}
\een
When $2^j\leq R_0$, it holds that
$
\|\Lambda  \ddj a\|_{L^2}\leq R_0 \|   \ddj a\|_{L^2}.
$
Thus, we could find a $\delta>0$ (small enough) such that 
\beno
\|\ddj a\|_{L^2}^2+(1-\delta)\|\ddj d\|^2_{L^2}+\delta \|\ddj w\|_{L^2}^2\geq \f1C( \|\ddj a\|_{L^2}^2+ \|\ddj d\|^2_{L^2} ).
\eeno
Integrating \eqref{est:a-d-aux1} over $[0, T]$, we get that
\beno
&& \|\ddj a\|_{L^2}+\|\ddj d\|_{L^2} +2^{2j }\int^T_0(  \|  \ddj d\|_{L^2}+ \delta\|  \ddj a\|_{L^2}  )\,dt\\
&&\leq C \int^T_0\|\div \cQ  u\|_{L^\infty}\|\ddj a\|_{L^2}\,dt+C\int^T_0\bigg(\| [\ddj ,u\cdot \na ] a\|_{L^2}+\|\ddj(a\div \cQ u )  \|_{L^2}\bigg)\,dt\\
&&\quad+ C\int^T_0\|\div \cQ u\|_{L^\infty}\|\ddj d\|_{L^2}\,dt+C\int^T_0 \bigg(\|[\ddj, u\cdot \na] d\|_{L^2}+\|\ddj \cQ W_\cQ\|_{L^2}\bigg)\,dt \\
&&\quad+C\int^T_0 \bigg(  \| \ddj[\Lambda, u\cdot \na ] a \|_{L^2}  +\|\ddj [\Lambda^{-1} \div, u\cdot \na] \cQ u\|_{L^2}+\|[\ddj, u\cdot \na] \Lambda a \|_{L^2}\bigg)\,dt  ,
\eeno
where $C$ depends on the $\mu,\lambda$ and $R_0$. By the definition of Besov space, we obtain
\beno
&& \| a\|^{\ell}_{\dot B^{\f12}_{2,1}}+\|  d\|^{\ell}_{\dot B^{\f12}_{2,1}} + \int^T_0(  \|   d\|^{\ell}_{\dot B^{\f52}_{2,1}}+\delta \|  a\|^{\ell}_{\dot B^{\f52}_{2,1}} )\,dt\\
&&\leq C \int^T_0\|\div \cQ  u\|_{L^\infty}\| a\|^{\ell}_{\dot B^{\f12}_{2,1}}\,dt+C\int^T_0\bigg(\sum_{2^j\leq R_0} 2^{\f{j}2}\| [\ddj ,u\cdot \na ] a\|_{L^2}+\| a\div \cQ u   \|^{\ell}_{\dot B^{\f12}_{2,1}} \bigg)\,dt\\
&&\quad+ C\int^T_0\|\div \cQ u\|_{L^\infty}\| d\|^{\ell}_{\dot B^{\f12}_{2,1}}\,dt+C\int^T_0 \bigg(\sum_{2^j\leq R_0} 2^{\f{j}2} \|[\ddj, u\cdot \na] d\|_{L^2}+\| \cQ W_\cQ\|^{\ell}_{\dot B^{\f12}_{2,1}}\bigg)\,dt\\
&&\quad+C\int^T_0 \sum_{2^j\leq R_0} 2^{\f{j}2} \|\ddj [\Lambda^{-1} \div, u\cdot \na] \cQ u\|_{L^2} \,dt  +C\int^T_0 \sum_{2^j\leq R_0} 2^{\f{j}2}   \| \ddj [\Lambda, u\cdot \na ]a \|_{L^2} \,dt\\
&&\quad+C\int^T_0 \sum_{2^j\leq R_0} 2^{\f{j}2}   \|[\ddj, u\cdot \na]  \Lambda a \|_{L^2}\,dt.
\eeno

 Let us give the estimates to the terms in the righthand side.
Due to Proposition \ref{prop:com} and Lemma \ref{Lem:com-est}, we  deduce that
\ben\label{est:com-low1}
\begin{split}
\sum_{2^j\leq R_0} 2^{\f{j}2}\| [\ddj ,u\cdot \na ] a\|_{L^2}&\leq \sum_{2^j\leq R_0} 2^{\f{j}2}\| [\ddj ,\cP u\cdot \na ] a\|_{L^2}+\sum_{2^j\leq R_0} 2^{\f{j}2}\| [\ddj ,\cQ u\cdot \na ] a\|_{L^2}\\
&\leq C  (\|\cQ u\| _{\dot B^{\f52,\f3p+1}_{2,p}}+\|\cP u\|_{\dot B^{\f3p+1}_{p,1}}) \|a\|_{\dot B^{\f12,\f 3p }_{2,p}}.
\end{split}
\een
The similar argument yields that
\beno
&&\sum_{2^j\leq R_0} 2^{\f{j}2} \|[\ddj, u\cdot \na] d\|_{L^2}\leq C (\|\cQ u\| _{\dot B^{\f52,\f3p+1}_{2,p}}+\|\cP u\|_{\dot B^{\f3p+1}_{p,1}})  \|d\| _{\dot B^{\f12,\f3p-1}_{2,p}} ,\\
&&\sum_{2^j\leq R_0} 2^{\f{j}2} \|\ddj [\Lambda^{-1} \div, u\cdot \na] \cQ u\|_{L^2} \leq C (\|\cQ u\| _{\dot B^{\f52,\f3p+1}_{2,p}}+\|\cP u\|_{\dot B^{\f3p+1}_{p,1}})   \|d\| _{\dot B^{\f12,\f3p-1}_{2,p}} ,\\
&&\sum_{2^j\leq R_0} 2^{\f{j}2}  \|\ddj [\Lambda, u\cdot \na ] a \|_{L^2}\leq  C (\|\cQ u\| _{\dot B^{\f52,\f3p+1}_{2,p}}+\|\cP u\|_{\dot B^{\f3p+1}_{p,1}})   \|a\| _{\dot B^{\f12,\f3p }_{2,p}} .
\eeno
Finally by  Proposition \ref{prop:pro}, we  have
\beno
\| a\div \cQ u\|^{\ell}_{\dot B^{\f12}_{2,1}}\leq C \|\cQ u\| _{\dot B^{\f52,\f3p+1}_{2,p}} \|a\|_{\dot B^{\f12,\f 3p }_{2,p}}.
\eeno

\smallskip

Now putting all the estimates and applying Lemma \ref{lem:G1}, we have
\ben\label{est:a1}
\begin{split}
&\quad\qquad \| a\|^{\ell}_{\dot B^{\f12}_{2,1}}+\|  d\|^{\ell}_{\dot B^{\f12}_{2,1}} + \int^T_0(  \|   d\|^{\ell}_{\dot B^{\f52}_{2,1}}+  \|  a\|^{\ell}_{\dot B^{\f52}_{2,1}} )\,dt \leq \| a_0\|^{\ell}_{\dot B^{\f12}_{2,1}}+\|  \cQ u_0\|^{\ell}_{\dot B^{\f12}_{2,1}}  \\
& \qquad\qquad+C \int_0^T \bigg(\|\cQ u\| _{\dot B^{\f52,\f3p+1}_{2,p}}+\|\cP u\|_{\dot B^{\f3p+1}_{p,1}}+\|a\|_{\dot B^{\f52,\f3p}_{2,p}}\bigg) (\|a\|_{\dot B^{\f12,\f 3p }_{2,p}}  +\|\cQ u\| _{\dot B^{\f12,\f3p-1}_{2,p}}+\|(\cP u)^h\|_{\dot B^{\f3p-1}_{p,1}}\bigg)\,dt.\end{split}
\een

\medskip

{\it Step 2: Estimates for the high frequency part of the solution.}
 Applying $\ddj$ on the both side of equation of $d,$ \eqref{equ:dlow}, and taking the inner product with $ |\ddj d|^{p-2}\ddj d$ , we derive that
 \beno
&&\f1p \f{d}{dt}\|\ddj d\|^p_{L^p}-(2\mu+\lambda)\int \Delta \ddj d \cdot |\ddj d|^{p-2}\ddj d\,dx  -\gamma\int \ddj \Lambda a\cdot    |\ddj d|^{p-2}\ddj d\,dx\\
&& \leq C\|\div \cQ u\|_{L^\infty}\|\ddj d\|^p_{L^p}+\|\ddj d\|^{p-1}_{L^p}\|[\ddj, u\cdot \na] d\|_{L^p}\nonumber\\
&&\quad +\|\ddj d\|^{p-1}_{L^p}(\|\ddj \cQ W_\cQ\|_{L^2}+\|\ddj [\Lambda^{-1} \div, u\cdot \na] \cQ u\|_{L^p}).\nonumber
 \eeno
 By using  Lemma A.5 and A.6 in \cite{Dan-CPDE}, we have that
  \ben\label{est:h1}
  \begin{split}
&\quad\f1p \f{d}{dt} \|\ddj d\|^p_{L^p}+(c_p 2^{2j}-1)\|\ddj d\|^p_{L^p}    \leq C \|\div \cQ u\|_{L^\infty}\|\ddj d\|^p_{L^p}+\|\ddj d\|^{p-1}_{L^p}\|[\ddj, u\cdot \na] d\|_{L^p}\\
&\quad\qquad +\|\ddj d\|^{p-1}_{L^p}(\|\ddj \cQ W_\cQ\|_{L^p}+\|\ddj [\Lambda^{-1} \div, u\cdot \na] \cQ u\|_{L^p}+\|\ddj w\|_{L^p}),
\end{split}
 \een
where $c_p$ is a positive constant depending only on $p.$

By the same argument applied to \eqref{equ:auxF}, we get that
  \ben\label{est:h2}
  \begin{split}
  & \f{d}{dt}\|\ddj w\|_{L^p}^p+\|\ddj w\|_{L^p}^p \\
  &\leq C\|\ddj d\|_{L^p}\|\ddj   w\|^{p-1}_{L^p}+ C\|\div u\|_{L^\infty}\|\ddj w\|_{L^p}^p +C\|[\ddj, u\cdot \na]w\|_{L^p}\|\ddj w\|^{p-1}_{L^p} \\
  &\quad+C\bigg(\|\ddj [\Lambda, u\cdot\na ] a\|_{L^p}+\|\ddj \Lambda(a\div u)\|_{L^p}+\|\ddj \Lambda^{-1} \div \cQ W_\cQ\|_{L^p}\bigg) \|\ddj w\|^{p-1}_{L^p}\\
  &\quad+C\|\ddj [\Lambda^{-1} \div, u\cdot \na] \cQ u\|_{L^p}\|\ddj w\|^{p-1}_{L^p}.\end{split}
  \een

Thanks to \eqref{est:h1} and \eqref{est:h2}, we have
\ben\label{est:h3}
\begin{split}
&\f1p \f{d}{dt}\bigg( \|\ddj d\|^p_{L^p}+\delta \|\ddj w\|_{L^p}^p\bigg)+(c_p 2^{2j}-2  )\|\ddj d\|^p_{L^p} + \delta \|\ddj w\|_{L^p}^p \\
& \leq 2\|\ddj d\|_{L^p}\|\ddj   w\|^{p-1}_{L^p}+ C \|\div \cQ u\|_{L^\infty}\|\ddj d\|^p_{L^p}+\|\ddj d\|^{p-1}_{L^p}\|[\ddj, u\cdot \na] d\|_{L^p} \\
&\quad +\|\ddj d\|^{p-1}_{L^p}\bigg(\|\ddj \cQ W_\cQ\|_{L^2}+\|\ddj [\Lambda^{-1} \div, u\cdot \na] \cQ u\|_{L^p} \bigg) +  C\|\div u\|_{L^\infty}\|\ddj w\|_{L^p}^p\\
&\quad  +C\bigg(\|[\ddj, u\cdot \na]w\|_{L^p}+\|\ddj [\Lambda^{-1} \div, u\cdot \na] \cQ u\|_{L^p}\bigg)\|\ddj w\|^{p-1}_{L^p} \\
  &\quad+C\bigg(\|\ddj [\Lambda, u\cdot\na ] a\|_{L^p}+\|\ddj \Lambda(a\div u)\|_{L^p}+\|\ddj \Lambda^{-1} \div \cQ W_\cQ\|_{L^p}\bigg) \|\ddj w\|^{p-1}_{L^p}.\end{split}
  \een

Observe that
 \beno
 \|\ddj d\|_{L^p}\|\ddj   w\|^{p-1}_{L^p}\leq \delta/2 \|\ddj  w\|^{p}_{L^p}+C_\delta  \|\ddj d\|^p_{L^p},
 \eeno
where the constant $\delta$ is chosen to satisfy the following inequality
 \beno
  \|\ddj d\|_{L^p}+\delta \|\ddj w\|_{L^p}\geq   \f12(\|\ddj d\|_{L^p}+\delta \|\ddj  \Lambda a \|_{L^p}).
 \eeno
Meanwhile, choose a suitable $R_0$ such that for any $2^j\geq R_0$ and $2\leq p\leq 4$, then we have  $c_p 2^{2j}-2-C_\delta\geq \f{c_p}2 2^{2j}$. Thus, when $2^j\geq R_0$, we  get that
  \ben\label{est:h4}
  \begin{split}
&  \f{d}{dt}\bigg( \|\ddj d\|_{L^p}+\delta \|\ddj \Lambda a\|_{L^p}\bigg)+ 2^{2j}\|\ddj d\|_{L^p} + \|\ddj \Lambda a\|_{L^p} \\
& \leq C \|\div \cQ u\|_{L^\infty}\|\ddj d\|_{L^p}+\|[\ddj, u\cdot \na] d\|_{L^p}  + \|\ddj \cQ W_\cQ\|_{L^p}+C\|\ddj \Lambda^{-1} \div \cQ W_\cQ\|_{L^p} \\
&\quad+\|\ddj [\Lambda^{-1} \div, u\cdot \na] \cQ u\|_{L^p} + C\|\div u\|_{L^\infty}\|\ddj \Lambda a\|_{L^p}+C\|[\ddj, u\cdot \na]\Lambda a\|_{L^p}\\
&\quad +C \|\ddj [\Lambda, u\cdot\na ] a\|_{L^p}+C\|\ddj \Lambda(a\div u)\|_{L^p}.\end{split}
\een
By the definition of Besov space, we deduce that
\beno
\begin{split}
&\|d, \Lambda a\|^H_{\dot B^{\f3p-1}_{p,1}}+\int_0^T \|d\|^H_{\dot B^{\f3p+1}_{p,1}}+\|\Lambda a\|^H_{\dot B^{\f3p-1}_{p,1}}\,dt\\
&\leq \int_0^T \|\div \cQ u\|_{L^\infty}\|d, \Lambda a\|^H_{\dot B^{\f3p-1}_{p,1}}\,dt +  \|  \cQ W_\cQ\|^H_{L^1_T(\dot B^{\f3p-1}_{p,1})}\\
&\quad+ \| a\div u \|^H_{L^1_T(\dot B^{\f3p }_{p,1})}+\int_0^T\sum_{2^j\geq R_0} 2^{j(3/p-1)}(\|[\ddj, u\cdot \na] d\|_{L^p}+ \|\ddj [\Lambda, u\cdot\na ] a\|_{L^p})\,dt\\
&\quad +\int_0^T\sum_{2^j\geq R_0} 2^{j(3/p-1)}(\|\ddj [\Lambda^{-1} \div, u\cdot \na] \cQ u\|_{L^p}+\|[\ddj, u\cdot \na] \Lambda a \|_{L^p}  )\,dt.
\end{split}
\eeno
Thanks to Proposition \ref{prop:pro}-Proposition \ref{prop:comp} as well as Lemma \ref{lem:G1}, the above inequality can be written by
\ben \label{est:a2}
\begin{split}
&\quad \|d, \Lambda a\|^H_{\dot B^{\f3p-1}_{p,1}}+\int_0^T\|d\|^H_{\dot B^{\f3p+1}_{p,1}}+\|\Lambda a\|^H_{\dot B^{\f3p-1}_{p,1}}\,dt \leq \|d_0, \Lambda a_0\|^H_{\dot B^{\f3p-1}_{p,1}}+C \int_0^T \bigg(\|\cQ u\| _{\dot B^{\f52,\f3p+1}_{2,p}}+\|\cP u\|_{\dot B^{\f3p+1}_{p,1}} \\
&\quad\quad+\|a\|_{\dot B^{\f52,\f3p}_{2,p}}\bigg)\bigg(\|a\|_{\dot B^{\f12,\f 3p }_{2,p}}  +\|\cQ u\| _{\dot B^{\f12,\f3p-1}_{2,p}}+\|(\cP u)^h\|_{\dot B^{\f3p-1}_{p,1}}\bigg)\,dt. 
\end{split}
\een

{\it Step 3: Estimates for incompressible part of the velocity.}
To close the estimates, we need to estimate   $\cP u.$  For $(\cP u)^h$,
Applying Proposition \ref{stokesprop} to the third equation of \eqref{neweq}, we obtain that
\ben\label{est:a3}
\begin{split}
&\|(\cP u)^h\|_{\wt L^\infty_T(\dot B^{\f3p-1}_{p,1})}+\|(\cP u)^h\|_{L^1_T(\dot B^{\f3p+1}_{p,1})} \\
&\leq\|(\cP u_0)^h\|_{\wt L^\infty_T(\dot B^{\f3p-1}_{p,1})}+ C\int_0^T\|\div u\|_{L^\infty}\|\cP u^h\|_{ \dot B^{\f3p-1}_{p,1} }\,dt+\int_0^T\|(\cP W_\cP)^h\|_{ \dot B^{\f3p-1}_{p,1} }\,dt \\
&\quad+C\int_0^T\sum_{j\in\N} 2^{j(\f 3p-1)}(\|[\ddj, u\cdot \na](\cP u)^h\|_{L^p}+\|[\cP, u\cdot\na ]u^h \|_{L^p})\,dt\\
&\leq\|(\cP u_0)^h\|_{\wt L^\infty_T(\dot B^{\f3p-1}_{p,1})}+ C\int_0^T\|a\|_{\dot B^{\f12,\f3p}_{2,p}}  (\|(\cQ u)^h\|_{\dot B^{5/2,3/p+1}_{2,p}}+\|(\cP u)^h\|_{ \dot B^{\f3p+1}_{p,1} }) \,dt \\
&\quad+C\int_0^T(\|(\cQ u)^h\|_{\dot B^{\f12,\f3p-1}_{2,p}}+\|(\cP u)^h\|_{\dot B^{\f3p-1}_{2,p}})  (\|\cQ u\|_{\dot B^{\f52,\f3p+1}_{2,p}}+\|\cP u \|_{ \dot B^{\f3p+1}_{p,1} })\,dt,
\end{split}
\een
where we used Lemma \ref{lem:G1} and
\beno
 &&\sum_{j\in\N} 2^{j(\f 3p-1)}(\|[\ddj, u\cdot \na](\cP u)^h\|_{L^p}+\|[\cP, u\cdot\na ]u^h \|_{L^p}) \\
  &&\leq C(\|(\cQ u)^h\|_{\dot B^{\f12,\f3p-1}_{2,p}}+\|(\cP u)^h\|_{\dot B^{\f3p-1}_{p,1}})  (\|\cQ u \|_{\dot B^{\f52,\f3p+1}_{2,p}}+\|\cP u \|_{ \dot B^{\f3p+1}_{p,1} }).
\eeno

For $(\cP u)^3$, deducing from divergence free $\div \cP u=0$, we  have
\beno
 \cP(u\cdot \na u^3 )= \cP\big(((\cQ u)^h+(\cP  u)^h)\pa_h  u^3 \big)-\cP(u^3 \div_h (\cP u)^h )+ \cP(u^3 \pa_3(\cQ u)^3 ).
\eeno
Thus, applying Proposition \ref{stokesprop} again to the last equation of \eqref{neweq}, we obtain that
\ben\label{est:a4}
\begin{split}
&\|(\cP u)^3\|_{\wt L^\infty_T(\dot B^{\f3p-1}_{p,1})}+\|(\cP u)^3\|_{L^1_T(\dot B^{\f3p+1}_{p,1})}\\
&\leq\|(\cP u_0)^3\|_{\wt L^\infty_T(\dot B^{\f3p-1}_{p,1})}+\int_0^T\|((\cQ u)^h+(\cP u)^h)\pa_h  u^3\|_{ \dot B^{\f3p-1}_{p,1} }\,dt \\
&\quad+\int_0^T\bigg(\|u^3 \div_h (\cP u)^h\|_{ \dot B^{\f3p-1}_{p,1} }+\|u^3 \pa_3(\cQ u)^3\|_{ \dot B^{\f3p-1}_{p,1} }+\|(\cP W_\cP)^3\|_{ \dot B^{\f3p-1}_{p,1} }\bigg)\,dt. \end{split}
\een
Thanks to Proposition \ref{prop:pro} and Lemma \ref{Lem:bi-est}, we have  
\beno
\|((\cQ u)^h+(\cP u)^h)\pa_h  u^3\|_{ \dot B^{\f3p-1}_{p,1} }
\leq C(\|(\cQ u)^h\|_{\dot B^{\f 12, \f3p-1}_{2,p}}+\|(\cP u)^h\|_{ \dot B^{\f3p-1}_{p,1} }) \|(\cP u)^3\|_{ \dot B^{\f3p+1}_{p,1} },
\eeno
and
\beno
\begin{split}
&\|u^3 \div_h (\cP u)^h\|_{ \dot B^{\f3p-1}_{p,1} }+ \|u^3 \pa_3(\cQ u)^3\|_{ \dot B^{\f3p-1}_{p,1} }+ \|(\cP W_\cP)^3\|_{ \dot B^{\f3p-1}_{p,1} }\\
 &\leq C(\|(\cQ u)^3\|_{\dot B^{\f 12, \f3p-1}_{2,p}}+\|(\cP u)^3\|_{ \dot B^{\f3p-1}_{p,1} })  (\|(\cP u)^h\|_{ \dot B^{\f3p+1}_{p,1} }+\|(\cQ u)^3\|_{ \dot B^{\f52, \f3p+1}_{2, p} } )\\
 &\quad+C \|a\|_{\dot B^{\f 12, \f3p }_{2,p}}  (\|(\cP u)^h\|_{ \dot B^{\f3p+1}_{p,1} }+\|(\cP u)^3\|_{ \dot B^{\f3p+1}_{p,1} } ),
 \end{split}
\eeno
from which together with Lemma \ref{lem:G1} and \eqref{est:a4}, we derive that
\ben\label{est:a5}
\begin{split}
&\|(\cP u)^3\|_{\wt L^\infty_T(\dot B^{\f3p-1}_{p,1})}+\|(\cP u)^3\|_{L^1_T(\dot B^{\f3p+1}_{p,1})}\\
&\leq\|(\cP u_0)^3\|_{\wt L^\infty_T(\dot B^{\f3p-1}_{p,1})}+ C\int_0^T (\|(\cQ u)^h\|_{\dot B^{\f 12, \f3p-1}_{2,p}}+\|(\cP u)^h\|_{ \dot B^{\f3p-1}_{p,1} })  \|(\cP u)^3\|_{ \dot B^{\f3p+1}_{p,1} } \,dt \\
&\quad+C\int^T_0(\|(\cQ u)^3\|_{\dot B^{\f 12, \f3p-1}_{2,p}}+\|(\cP u)^3\|_{ \dot B^{\f3p-1}_{p,1} })  (\|(\cP u)^h\|_{ \dot B^{\f3p+1}_{p,1} }+\|(\cQ u)^3\|_{ \dot B^{\f52, \f3p+1}_{2, p} } )\,dt\\
&\quad+C\int^T_0 \|a\|_{\dot B^{\f 12, \f3p }_{2,p}}  (\|(\cP u)^h\|_{ \dot B^{\f3p+1}_{p,1} }+\|(\cP u)^3\|_{ \dot B^{\f3p+1}_{p,1}} )\,dt.\end{split}
\een

 {\it Step 4: Continuity argument.}  
We first deduce from \eqref{est:a1}, \eqref{est:a2} and \eqref{est:a3} that   
\beno
&&\|a(t)\|_{\dot B^{\f12, \f3p}_{2,p}}+ \|d(t)\|_{\dot B^{\f12, \f3p-1}_{2,p}}+ \|(\cP u)^h(t)\|_{\dot B^{  \f3p-1}_{p,1}}+\int^t_0\bigg(\|a\|_{\dot B^{\f52, \f3p}_{2,p}}+ \|d\|_{\dot B^{\f52, \f3p+1}_{2,p}}+ \|\cP u \|_{\dot B^{ \f3p+1}_{p,1}}\bigg)d\tau\\
&&\leq C\bigg(\|a_0\|_{\dot B^{\f12, \f3p}_{2,p}}+ \|d_0\|_{\dot B^{\f12, \f3p-1}_{2,p}}+ \|(\cP u_0)^h\|_{\dot B^{ \f3p-1}_{p,1}}\bigg)\\
&&\quad +\int^t_0\bigg(\|a\|_{\dot B^{\f52, \f3p}_{2,p}}+ \|d\|_{\dot B^{\f52, \f3p+1}_{2,p}}+ \|\cP u \|_{\dot B^{ \f3p+1}_{p,1}}\bigg)\bigg(\|a\|_{\dot B^{\f12, \f3p}_{2,p}}+ \|d\|_{\dot B^{\f12, \f3p-1}_{2,p}}+ \|(\cP u)^h\|_{\dot B^{  \f3p-1}_{p,1}}\bigg)\,d\tau.
\eeno 
 
Let $\eps^{\f12}\le c_1\ll1$. We define $\bar{T}$ by
\begin{equation}\label{defit}
\begin{split}
\bar{T}\eqdefa\sup \{ T>0: \|a\|_{\wt{L}^\infty_T(\dot{B}^{\f12, \f3p}_{2,p})} +\|d\|_{\wt{L}^\infty_T(\dot{B}^{\f12, \f3p-1}_{2,p})}+\|(\cP u)^h\|_{\wt{L}^\infty_T(\dot{B}^{\f3p-1}_{p,1})}  \leq c_1 \}.
\end{split}
\end{equation}
According to the local existence and blow up criterion for the system, it is obvious that $\bar{T}>0.$ We shall prove $\bar{T}=\infty$ under the assumption \eqref{eq:smalldata1}. 
 For any $t\in [0, \bar{T}],$ the above inequality can be recast by
 \ben\label{adu}
 &&\|a(t)\|_{\dot B^{\f12, \f3p}_{2,p}}+ \|d(t)\|_{\dot B^{\f12, \f3p-1}_{2,p}}+ \|(\cP u)^h(t)\|_{\dot B^{  \f3p-1}_{p,1}}+(1-c_1)\int^t_0\bigg(\|a\|_{\dot B^{\f52, \f3p}_{2,p}}+ \|d\|_{\dot B^{\f52, \f3p+1}_{2,p}}+ \|(\cP u)^h \|_{\dot B^{ \f3p+1}_{p,1}}\bigg)d\tau\notag\\
 &&\leq  \|a_0\|_{\dot B^{\f12, \f3p}_{2,p}}+ \|d_0\|_{\dot B^{\f12, \f3p-1}_{2,p}}+ \|(\cP u_0)^h\|_{\dot B^{ \f3p-1}_{p,1}}+C\int^t_0  \|(\cP u)^3 \|_{\dot B^{ \f3p+1}_{p,1}}\bigg(\|a\|_{\dot B^{\f12, \f3p}_{2,p}}+ \|d\|_{\dot B^{\f12, \f3p-1}_{2,p}}\notag\\&&\qquad\qquad+ \|(\cP u)^h\|_{\dot B^{  \f3p-1}_{p,1}}\bigg)\,d\tau,
 \een
which implies that 
 \ben\label{localest1}
 \begin{split}
 \|a(t)\|_{\dot B^{\f12, \f3p}_{2,p}}+ \|d(t)\|_{\dot B^{\f12, \f3p-1}_{2,p}}+ \|(\cP u)^h(t)\|_{\dot B^{  \f3p-1}_{p,1}}\leq& C\bigg(\|a_0\|_{\dot B^{\f12, \f3p}_{2,p}}+ \|d_0\|_{\dot B^{\f12, \f3p-1}_{2,p}}+ \|(\cP u_0)^h\|_{\dot B^{ \f3p-1}_{p,1}}\bigg)\\
&\times\exp \bigg( {\int^t_0  \|(\cP u)^3 \|_{\dot B^{ \f3p+1}_{p,1}}}\,d\tau\bigg),
\end{split}
\een
for any $t\in [0, \bar{T}].$

On the other hand,  from \eqref{est:a5} and \eqref{defit} ,  we  get that for any $t\in [0, \bar{T}],$
\beno
\begin{split}
\|(\cP u)^3\|_{\wt L^\infty_t(\dot B^{\f3p-1}_{p,1})}+\f23\|(\cP u)^3\|_{L^1_t(\dot B^{\f3p+1}_{p,1})}
\leq& \|(\cP u_0)^3\|_{\dot B^{\f3p-1}_{p,1}}+\int^t_0 \bigg(\|a\|_{\dot B^{\f52, \f3p}_{2,p}}+ \|d\|_{\dot B^{\f52, \f3p+1}_{2,p}}+ \|(\cP u)^h \|_{\dot B^{ \f3p+1}_{p,1}}\bigg)  \\
&\times\bigg(\|a\|_{\dot B^{\f12, \f3p}_{2,p}}+ \|d\|_{\dot B^{\f12, \f3p-1}_{2,p}}+ \|(\cP u)^h\|_{\dot B^{  \f3p-1}_{p,1}}\bigg)\,d\tau,
\end{split}
\eeno
 from which together with \eqref{defit} and \eqref{adu}, we obtain that for any $t\in [0, \bar{T}],$
\ben\label{intepu3}  {\int^t_0  \|(\cP u)^3 \|_{\dot B^{ \f3p+1}_{p,1}}}\,d\tau\le C( \|a_0\|_{\dot B^{\f12, \f3p}_{2,p}}+ \|d_0\|_{\dot B^{\f12, \f3p-1}_{2,p}}+ \|(\cP u_0)^h\|_{\dot B^{ \f3p-1}_{p,1}}+\|(\cP u_0)^3\|_{\dot B^{\f3p-1}_{p,1}}). \een
 Plugging this estimates into \eqref{localest1} and using the condition \eqref{eq:smalldata1}, we obtain that
\ben\label{adpuest}
\begin{split}
&\|a(t)\|_{\dot B^{\f12, \f3p}_{2,p}}+ \|d(t)\|_{\dot B^{\f12, \f3p-1}_{2,p}}+ \|(\cP u)^h(t)\|_{\dot B^{  \f3p-1}_{p,1}}\\
&\leq C\bigg(\|a_0\|_{\dot B^{\f12, \f3p}_{2,p}}+ \|d_0\|_{\dot B^{\f12, \f3p-1}_{2,p}}+ \|(\cP u_0)^h\|_{\dot B^{ \f3p-1}_{p,1}}\bigg)  \exp{\big(C (\eps+\|(\cP u_0)^3\|_{ \dot B^{\f3p-1}_{p,1}})\big)}\\
&\leq C\eps\leq \f{c_1}2,
\end{split}
\een
for any $t\in [0, \bar{T}],$ which contradicts with the definition of $\bar{T}$. We conclude that $\bar{T}=\infty$ and \eqref{adpuest} holds for all time, from which together with \eqref{adu} and \eqref{intepu3} will imply the desired estimates in the theorem. It ends the proof of Theorem \ref{thm:main3}. 
\end{proof}

\setcounter{equation}{0}
  \section{Appendix}
For  readers' convenience, in this appendix, we list some basic knowledge on  Littlewood-Paley  theory.
\subsection{Littlewood-Paley decomposition}

Let us introduce the Littlewood-Paley decomposition. Choose a
radial function  $\varphi \in {\cS}(\R^3)$ supported in
${\cC}=\{\xi\in\R^3,\, \frac{3}{4}\le|\xi|\le\frac{8}{3}\}$ such
that
\beno \sum_{j\in\Z}\varphi(2^{-j}\xi)=1 \quad \textrm{for
all}\,\,\xi\neq 0. \eeno The frequency localization operator
$\Delta_j$ and $S_j$ are defined by
\begin{align}
\Delta_jf=\varphi(2^{-j}D)f,\quad S_jf=\sum_{k\le
j-1}\Delta_kf\quad\mbox{for}\quad j\in \Z. \nonumber
\end{align}
With our choice of $\varphi$, one can easily verify that
\ben\label{orth}
\begin{aligned}
\Delta_j\Delta_kf=0\quad \textrm{if}\quad|j-k|\ge 2\quad
\textrm{and}
\quad
\Delta_j(S_{k-1}f\Delta_k f)=0\quad \textrm{if}\quad|j-k|\ge 5.
\end{aligned}
\een

Next we recall Bony's decomposition from \cite{Bony}:
\ben\label{Bonydecom}
uv=T_uv+T_vu+R(u,v),
\een
with
$$T_uv=\sum_{j\in\Z}S_{j-1}u\Delta_jv, \quad R(u,v)=\sum_{j\in\Z}\Delta_ju \widetilde{\Delta}_{j}v,
\quad \widetilde{\Delta}_{j}v=\sum_{|j'-j|\le1}\Delta_{j'}v.$$

\subsection{Product estimates in Besov spaces}

We first recall the Bernstein lemma which will be frequently used(see \cite{Bah}).

\begin{lem}\label{Lem:Bernstein}
Let $1\le p\le q\le+\infty$. Assume that $f\in L^p(\R^3)$, then for
any $\gamma\in(\N\cup\{0\})^3$, there exist constants $C_1$, $C_2$
independent of $f$, $j$ such that
\beno
&&{\rm supp}\hat f\subseteq
\{|\xi|\le A_02^{j}\}\Rightarrow \|\partial^\gamma f\|_q\le
C_12^{j{|\gamma|}+3j(\frac{1}{p}-\frac{1}{q})}\|f\|_{p},
\\
&&{\rm supp}\hat f\subseteq \{A_12^{j}\le|\xi|\le
A_22^{j}\}\Rightarrow \|f\|_{p}\le
C_22^{-j|\gamma|}\sup_{|\beta|=|\gamma|}\|\partial^\beta f\|_p.
\eeno
\end{lem}
 
As a consequence,  the    estimates  for paraproduct and remainder operators can be given by 
 \begin{lem}\label{Lem:bonyweighttimespace}
Let $1\le p,q,q_1,q_2\le \infty$ with $\f 1{q_1}+\f1{q_2}=\f1q$.
Then there hold

 (a)\, if $s_2\le \frac{3}{p}$, we have
\beno
\|T_gf\|_{\widetilde{L}^{q}_T(\dot B^{s_1+s_2-
\frac{3}{p}}_{p,1} )} \le C\|f\|_{\widetilde{L}^{q_1}_T(\dot
B^{s_1}_{p,1} )}\|g\|_{\widetilde{L}^{q_2}_T(\dot
B^{s_2}_{p,1})}; \eeno

 (b)\, if $s_1\le \frac{3}{p}-1$, we have
\beno
\|T_fg\|_{\widetilde{L}^{q}_T(\dot B^{s_1+s_2-
\frac{3}{p}}_{p,1} )} \le C\|f\|_{\widetilde{L}^{q_1}_T(\dot
B^{s_1}_{p,1} )}\|g\|_{\widetilde{L}^{q_2}_T(\dot
B^{s_2}_{p,1})};
\eeno

  (c)\, if $s_1+s_2>3\max(0,\frac 2p-1)$, we have
\beno
\|R(f,g)\|_{\widetilde{L}^{q}_T(\dot B^{s_1+s_2-
\frac{3}{p}}_{p,1} )} \le C\|f\|_{\widetilde{L}^{q_1}_T(\dot
B^{s_1}_{p,1} )}\|g\|_{\widetilde{L}^{q_2}_T(\dot
B^{s_2}_{p,1})}.
\eeno
\end{lem}
We refer readers  to \cite{Bah} for detailed proof.
Then we arrive at the  product estimates:
\begin{lem}\label{Lem:bi-weight}
Let $s_1\le \frac{3}{p}-1, s_2\le \frac{3}{p}, s_1+s_2>3\max(0,\frac 2p-1)$, and $1\le p,q,q_1,q_2\le \infty$
with $\f 1{q_1}+\f1{q_2}=\f1q$. Then there holds
\beno
\|fg\|_{\widetilde{L}^{q}_T(\dot B^{s_1+s_2-
\frac{3}{p}}_{p,1} )} \le C\|f\|_{\widetilde{L}^{q_1}_T(\dot
B^{s_1}_{p,1} )}\|g\|_{\widetilde{L}^{q_2}_T(\dot
B^{s_2}_{p,1})}.
\eeno
\end{lem}
 
\begin{lem}\label{Lem:bi-est}
Let $s_1, s_2\le \frac{3}{p},\, s_1+s_2>3\max (0,\frac2p-1)$, and $1\le p,q,q_1,q_2\le \infty$ with $\f 1{q_1}+\f1{q_2}=\f1q$.
Then there holds
\beno \|fg\|_{\widetilde{L}^{q}_T(\dot B^{s_1+s_2-
\frac{3}{p}}_{p,1})} \le C\|f\|_{\widetilde{L}^{q_1}_T(\dot
B^{s_1}_{p,1})}\|g\|_{\widetilde{L}^{q_2}_T(\dot B^{s_2}_{p,1})}.
\eeno
\end{lem}

Next  commutator between the frequency localization operator and the function  can be estimated by:
\begin{lem}\label{Lem:com-est}
Let $p\in [1,\infty)$ and $s\in (-3\min(\frac 1p,\frac1{p'}),\frac 3p]$.
Then it hold
\beno
\big\|2^{js} \|[\Delta_j,f]\na
g\|_{L^1_T(L^p)}\big\|_{\ell^1}\le C\|f\|_{\widetilde{L}^\infty_T(\dot B^{
\frac{3}{p}}_{p,1})}\|g\|_{L^1_T(\dot
B^{s+1}_{p,1})},\\
\big\|2^{js} \|[\Delta_j,f]\na
g\|_{L^1_T(L^p)}\big\|_{\ell^1}\le C\|f\|_{\widetilde{L}^\infty_T(\dot B^{
\frac{3}{p}+1}_{p,1})}\|g\|_{L^1_T(\dot
B^{s}_{p,1})}.
\eeno
\end{lem}

We recall the following composite result:
\begin{lem}\label{Lem:non-est}
Let $s>0$ and $1\le p,q,r\le \infty$. Assume that $F\in W^{[s]+3,\infty}_{loc}(\R)$ with  $F(0)=0$.
Then there holds
\beno
\|F(f)\|_{\widetilde{L}^q_T(\dot B^s_{p,r} )}
\le C(1+\|f\|_{L^\infty_T(L^\infty)})^{[s]+2}\|f\|_{\widetilde{L}^q_T(\dot B^s_{p,r} )}.
\eeno
\end{lem}

 In the anisotropic spaces, the above results still hold ture. We refer readers to   \cite{Charve} and \cite{CMZ} for details. We begin with the product estimates:
 
  \begin{prop}\label{prop:pro}
 Let $s, t, \wt s, \wt t, \sigma, \tau\in \mathbb{R}, 2\leq p\leq 4$, and $1\leq r, r_1, r_2\leq \infty$ with $\f 1r=\f1{r_1}+\f1{r_2}$. Then we have the following estimates:

  (a)\, If $\sigma, \tau \leq \f n p$ and $\sigma+\tau>0$, then
  \beno
  &&\sum_{2^j>R_0} 2^{j(\sigma+\tau-\f np)}\|\ddj(f g)\|_{L^r_T (L^p)}\\
  &&\leq
  C(\|f\|^\ell_{\wt L^{r_1}_T(\dot B^{n/2-n/p+\sigma}_{2, 1})}+\|f\|^H_{\wt L^{r_1}_T(\dot B^{\sigma}_{p, 1})})(\|g\|^\ell_{\wt L^{r_2}_T(\dot B^{n/2-n/p+\tau}_{2, 1})}+\|g\|^H_{\wt L^{r_2}_T(\dot B^{\tau}_{p, 1})}).
  \eeno

   (b)\, If $s,\wt s\leq \f np$ and $s+t> n-\f {2n}p$ with $s+t=\wt s+\wt t$ and $\theta\in \mathbb{R}$, then
     \beno
  &&\sum_{2^j\leq R_0} 2^{j(s+t-\f n2)}\|\ddj(f g)\|_{L^r_T (L^2)}\\
  &&\leq
 C(\|f\|^\ell_{\wt L^{r_1}_T(\dot B^{s}_{2, 1})}+\|f\|^H_{\wt L^{r_1}_T(\dot B^{s-n/2+n/p}_{p, 1})})(\|g\|^\ell_{\wt L^{r_2}_T(\dot B^{t}_{2, 1})}+\|g\|^H_{\wt L^{r_2}_T(\dot B^{t-n/2+n/p+\theta}_{p, 1})})\\
&& \quad+C(\|g\|^\ell_{\wt L^{r_2}_T(\dot B^{\wt s}_{2, 1})}+\|g\|^H_{\wt L^{r_2}_T(\dot B^{\wt s-n/2+n/p}_{p, 1})})(\|f\|^\ell_{\wt L^{r_1}_T(\dot B^{\wt t}_{2, 1})}+\|f\|^H_{\wt L^{r_1}_T(\dot B^{\wt t-n/2+n/p}_{p, 1})}).
  \eeno

  (c)\, If $s,\wt s\leq \f n2$ and $s+t>\f n2-\f np$ with $s+t=\wt s+\wt t$, then
       \beno
  &&\sum_{j\in \mathbb{Z}} 2^{j(s+t-\f n2)}\|\ddj(f g)\|_{L^r_T (L^2)} 
 \leq
 C(\|f\|^\ell_{\wt L^{r_1}_T(\dot B^{s}_{2, 1})}+\|f\|^H_{\wt L^{r_1}_T(\dot B^{s-n/2+n/p}_{p, 1})})\|g\|_{\wt L^{r_2}_T(\dot B^{t}_{2, 1})} \\
&& \quad +C(\|g\|^\ell_{\wt L^{r_2}_T(\dot B^{\wt s}_{2, 1})}+\|g\|^H_{\wt L^{r_2}_T(\dot B^{\wt s-n/2+n/p}_{p, 1})})\|f\| _{\wt L^{r_1}_T(\dot B^{\wt t}_{2, 1})} .
  \eeno
\end{prop}

 Let $\{c(j)\}$ be a sequence in $l^1$ with the norm $\|\{c(j)\}\|_{l^1}\leq 1$. Then  the estimates for  the commutators can be stated as
\begin{prop}\label{prop:com}
Let $2\leq p\leq 4, -\f np <s\leq \f n2+1$, $-\f np <\sigma\leq \f np+1$, and $1\leq r, r_1, r_2\leq\infty$ with $\f1 r=\f1{r_1}+\f1{r_2}$ . Then for $2^j>R_0$, it hold
\beno
&&\|[v, \ddj]\cdot \na f\|_{L^r_T (L^p)} \leq C c(j)(2^{-j\sigma}+2^{j(\f n2-\f np-s)})(\|f\|^\ell_{\wt L^{r_2}_T(\dot B^{s}_{2, 1})}+\|f\|^H_{\wt L^{r_2}_T(\dot B^{\sigma}_{p, 1})})\\&&\qquad
\times(\|v\|^\ell_{\wt L^{r_1}_T(\dot B^{n/2+1}_{2, 1})}+\|v\|^H_{\wt L^{r_1}_T(\dot B^{1+n/p}_{p, 1})}),\\
&&\|[v, \ddj]\cdot \na f\|_{L^r_T (L^p)} \leq C c(j)(2^{-j\sigma}+2^{j(\f n2-\f np-s)})(\|f\|^\ell_{\wt L^{r_2}_T(\dot B^{s}_{2, 1})}+\|f\|^H_{\wt L^{r_2}_T(\dot B^{\sigma}_{p, 1})}) \|v\|_{\wt L^{r_1}_T(\dot B^{1+n/p}_{p, 1})}.
\eeno
Moreover, if $-\f n p<s\leq \f np+1$, then
\beno
&&\|[v, \ddj]\cdot \na f\|_{L^r_T (L^2)} \leq C  c(j) 2^{-js} \|f\|_{\wt L^{r_2}_T(\dot B^{s}_{2, 1})} (\|v\|^\ell_{\wt L^{r_1}_T(\dot B^{n/2+1}_{2, 1})}+\|v\|^H_{\wt L^{r_1}_T(\dot B^{1+n/p}_{p, 1})}).
\eeno
 \end{prop}
 
 \begin{prop}\label{prop:com1}
 Under the assumption of Proposition \ref{prop:com},  if $S\in S^{m}_{1,0}$,  then  for $2^j>R_0$,
\beno
&&\|[S, \ddj]\cdot \na f\|_{L^r_T (L^p)} \leq C c(j)(2^{-j\sigma}+2^{j(\f n2-\f np-s)})(\|f\|^\ell_{\wt L^{r_2}_T(\dot B^{s+m}_{2, 1})}+\|f\|^H_{\wt L^{r_2}_T(\dot B^{\sigma+m}_{p, 1})}).
\eeno
And for $2^j\leq R_0$,
if $-\f n p<s\leq \f np+1$, then
\beno
&&\|[S, \ddj]\cdot \na f\|_{L^r_T (L^2)} \leq C  c(j) 2^{-js} (\|f\|^\ell_{\wt L^{r_2}_T(\dot B^{s+m}_{2, 1})}+\|f\|^H_{\wt L^{r_2}_T(\dot B^{\sigma+m}_{p, 1})}) ,
\eeno
 \end{prop}

Finally the composite result can be proved:
 \begin{prop}\label{prop:comp}
 Let $2\leq p\leq 4, s, \sigma>0,$ and $s\geq \sigma-\f n2+\f np, r\geq 1$. Assume that $F\in W^{[s]+2}_{loc}\cap W^{[\sigma]+2}_{loc}$ with $F(0)=0$. Then there holds
 \beno
 &&\|F(f)\|^\ell_{\wt L^r_T(\dot B^s_{2,1})}+ \|F(f)\|^H_{\wt L^r_T(\dot B^\sigma_{p,1})}\\
 &&\leq C(1+\|f\|^\ell_{\wt L^\infty_T(\dot B^{n/p}_{2,1})}+\|f\|^H_{\wt L^\infty_T(\dot B^{n/p}_{p,1})})^{\max([s],[\sigma])+1}(\|f\|^\ell_{\wt L^r_T(\dot B^s_{2,1})}+ \|f\|^H_{\wt L^r_T(\dot B^\sigma_{p,1})}).
 \eeno
 For any $s>0$ and $p\geq 1$, there holds
 \beno
 \|F(f)\|_{\wt L^r_T(\dot B^s_{p,1})}\leq C(1+\|f\|_{L^\infty_T(L^\infty)})^{[s]+1}\|f\|_{\wt L^r_T(\dot B^s_{p,1})}.
 \eeno
 \end{prop}

 For the  Stokes equations, the maximum regularity estimate can be concluded as
 \begin{prop}\label{stokesprop}
{\sl Let $p\in (1,\infty)$, and $s\in \mathbb{R}$.
Let $u_0\in \dot{B}^s_{p,1}(\R^3)$ be a divergence-free field and
$g\in\wt{L}^1_T(\dot{B}^s_{p,1})$. If $u$ solves
\begin{equation}\label{stokes}
 \left\{\begin{array}{l}
\displaystyle \pa_t u -\mu\Delta u+\na \Pi=g,\\
\displaystyle \div u = 0,\\
\displaystyle u|_{t=0}=u_0,
\end{array}\right.
\end{equation}
then \eqref{stokes} has a unique solution $u$ so that  \beno
\|u\|_{\wt{L}^\infty_T(\dot{B}^s_{p,1})}
+\mu\|u\|_{L^1_T(\dot{B}^{s+2}_{p,1})}+\|\na\Pi\|_{L^1_T(\dot{B}^s_{p,1})}
\leq \|u_0\|_{\dot{B}^s_{p,1}} + C\|g\|_{L^1_T(\dot{B}^s_{p,1})}. \eeno}
\end{prop}

\end{document}